\documentclass[11pt]{article}
\usepackage{amsmath}
\usepackage{amssymb}
\usepackage{amsfonts}
\usepackage{hyperref}
\usepackage[dvips]{graphicx}

\newtheorem{theorem}{Theorem}[section]

\newtheorem{corollary}[theorem]{Corollary}

\newtheorem{definition}[theorem]{Definition}

\newtheorem{lemma}[theorem]{Lemma}

\newtheorem{proposition}[theorem]{Proposition}

\newenvironment{proof}[1][Proof]{\noindent\textbf{#1.} }{\ \rule{0.5em}{0.5em}}
\newenvironment{acknowledgment}[1][Acknowledgement]{\textbf{#1} }{}
\renewcommand{\theequation}{\thesection.\arabic{equation}}
\input{tcilatex}
\begin{document}

\title{Short-time behaviour of a modified Laplacian coflow of $G_{2}$%
-structures}
\author{Sergey Grigorian \\
Simons Center for Geometry and Physics\\
Stony Brook University\\
Stony Brook, NY 11794\\
USA}
\maketitle

\begin{abstract}
We modify the Laplacian coflow of co-closed $G_{2}$-structures - $\frac{d}{dt%
}\psi =\Delta \psi $ where $\psi $ is the closed dual $4$-form of a $G_{2}$%
-structure $\varphi $. The modified flow is now weakly parabolic in the
direction of closed forms up to diffeomorphisms. We then prove short time
existence and uniqueness of solutions to the modified flow.
\end{abstract}

\section{Introduction}

\setcounter{equation}{0}Ever since the Ricci flow has been introduced by
Richard Hamilton \cite{Hamilton3folds}, geometric flows have played an
important role in the study of geometric structures. The general idea is to
begin with some general geometric structure on a manifold and then use some
flow to obtain a more special structure. In the study of $G_{2}$-structures
on $7$-dimensional manifolds, an important question is under what conditions
is it possible to obtain a torsion-free $G_{2}$-structure (which corresponds
to manifolds with holonomy contained in $G_{2}$) from a $G_{2}$-structure
with some other torsion class. To do this, one can either attempt a
non-infinitesimal deformation of the $G_{2}$-structure \cite%
{GrigorianG2Torsion1, GrigorianG2TorsionWarped,karigiannis-2005-57}, or one
can try to construct a flow which interpolates between different torsion
classes. The first such flow has been proposed by Bryant \cite{bryant-2003}
- if we start with a closed $G_{2}$-structure, that is one for which the
defining $3$-form $\varphi $ satisfies 
\begin{equation*}
d\varphi =0
\end{equation*}%
then the Laplacian $\Delta \varphi =dd^{\ast }\varphi $ is an exact form,
and hence Laplacian flow of $\varphi $ 
\begin{equation}
\frac{d\varphi }{dt}=\Delta _{\varphi }\varphi  \label{orig-flow}
\end{equation}%
preserves the cohomology class, and we get a flow of closed $G_{2}$%
-structures within the same cohomology class. Here $\Delta _{\varphi }$
denotes the Hodge Laplacian with respect to the metric $g_{\varphi }$
associated to the $G_{2}$-structure $\varphi $.

Suppose now $M$ is a compact $7$-dimensional manifold with a $G_{2}$%
-structure $\varphi $. It is then possible to interpret the flow (\ref%
{orig-flow}) as a gradient flow of the volume functional $V$ \cite{BryantXu} 
\begin{equation}
V\left( \varphi \right) =\frac{1}{7}\int_{M}\varphi \wedge \ast \varphi
\label{volfunc}
\end{equation}%
Then, the functional $V$ increases monotonically along this flow. As shown
by Hitchin \cite{Hitchin:2000jd-arxiv}, for a closed $G_{2}$-structure $%
\varphi $, $V$ attains a critical point within the fixed cohomology class of 
$\varphi $ whenever $d\ast \varphi =0$, that is, if and only if $\varphi $
defines a torsion-free $G_{2}$-structure. Therefore, it is to be expected
that if a long-time smooth solution to (\ref{orig-flow}) exists, then it
should converge to a torsion-free $G_{2}$-structure. In \cite{BryantXu,XuYe}
it is shown that after applying a version of the DeTurck Trick (named after
DeTurck's proof of short time existence and uniqueness of Ricci flow
solutions \cite{DeTurckTrick}), the flow (\ref{orig-flow}) can be related to
a modified flow that is parabolic along closed forms. This was then used to
prove short time existence and uniqueness of solutions. Moreover, in \cite%
{XuYe}, it was shown that if the initial closed $G_{2}$-structure $\varphi
_{0}$ is near a torsion-free structure $\varphi _{1}$ then the flow (\ref%
{orig-flow}) converges to a torsion-free $G_{2}$-structure $\varphi _{\infty
}$ which is related to $\varphi _{1}$ via a diffeomorphism.

Other flows of $G_{2}$-structures have also been proposed. In \cite%
{karigiannis-2007}, Karigiannis studied the properties of general flows of $%
\varphi $ by an arbitrary $3$-form. In \cite{WeissWitt1, WeissWitt2}, Weiss
and Witt made significant progress while studying gradient flows of
Dirichlet-type functionals for $G_{2}$-structures. They have obtained
short-time existence and uniqueness results as well as long-time convergence
to a torsion-free $G_{2}$-structure if the initial condition is sufficiently
close to a torsion-free $G_{2}$-structure, i.e. stability of the flow.

Most of the flows studied focused on flowing the $3$-form $\varphi $.
However, a $G_{2}\,$-structure can also be defined by the dual $4$-form $%
\ast \varphi $, which we will denote by $\psi .$ It is then natural to
consider the analogue of the Laplacian flow (\ref{orig-flow}), but for the $%
4 $-form. Such a flow,%
\begin{equation}
\frac{d\psi }{dt}=\Delta _{\psi }\psi  \label{orig-coflow}
\end{equation}%
named the Laplacian coflow of $G_{2}$-structures, has originally been
proposed by Karigiannis, McKay and Tsui in \cite{KarigiannisMcKayTsui}. Here 
$\Delta _{\psi }$ is the Hodge Laplacian defined by the metric $g_{\psi }$
which is the metric associated to the $G_{2}$-structure defined by $\psi $.
Note that in \cite{KarigiannisMcKayTsui}, there was a minus sign on the
right hand side of (\ref{orig-coflow}). The flow (\ref{orig-coflow}) shares
a number of properties with (\ref{orig-flow}). In particular, if we start
with $\psi $ closed, that is, a co-closed $G_{2}$-structure, then (\ref%
{orig-coflow}) preserves the cohomology class of $\psi $. Thus we get a flow
of co-closed $G_{2}$-structures. The volume functional \ can be restated in
terms of $\psi $, and then it attains a critical point within the cohomology
class of $\psi $ whenever $d\ast \psi =d\varphi =0$. The flow (\ref%
{orig-coflow}) can then also be interpreted as a gradient flow of the volume
functional, and it is easy to see that the volume grows monotonically along
this flow. While qualitatively some of the properties are similar to the
Laplacian flow on $3$-forms, the initial conditions are completely different
- in (\ref{orig-flow}) we start from a closed $G_{2}$-structure, while in (%
\ref{orig-coflow}), we start from a co-closed $G_{2}$-structure.

In this paper we study the analytical properties of the flow (\ref%
{orig-coflow}). It turns out that despite the similarities with the $3$-form
flow (\ref{orig-flow}), the $4$-form flow cannot be related to a flow that
is strictly parabolic in the direction of closed forms using
diffeomorphisms. Therefore, we propose a modified version of (\ref%
{orig-coflow}), given by 
\begin{equation}
\frac{d\psi }{dt}=\Delta _{\psi }\psi +2d\left( \left( A-\func{Tr}T\right)
\varphi \right)  \label{mod-coflow}
\end{equation}%
Here $\func{Tr}T$ is the trace of the full torsion tensor $T$ of the $G_{2}$%
-structure defined by $\psi $, and $A$ is a positive constant. This flow is
now weakly parabolic in the direction of closed forms and hence it is
possible to relate it to a strictly parabolic flow using an application of
deTurck's trick. The flow (\ref{mod-coflow}) still preserves the cohomology
class of $\psi $ and if $\func{Tr}T$ is small enough in some sense, the
volume functional grows along the flow. We then use the techniques from (%
\cite{BryantXu}) to show short-time existence and uniqueness for this flow.

The outline of the paper is as follows. In Section \ref{secg2struct} we give
a brief overview of the properties of $G_{2}$-structures and their torsion.
In Section \ref{SecDeformPsi} we consider the deformations of $G_{2}$%
-structures in terms of deformations of the $4$-form $\psi $, and in Section %
\ref{SecLapPsi} we consider the properties of $\Delta _{\psi }\psi $,
including its linearization. The modified flow (\ref{mod-coflow}) is then
defined in Section \ref{SecModFlow}, and in Section \ref{SecShortTime} we
prove the short time existence and uniqueness of solutions of (\ref%
{mod-coflow}). For that we adapt the techniques used by Bryant and Xu in 
\cite{BryantXu} for the Laplacian flow of $\varphi $.

\begin{acknowledgment}
I would like to thank the anonymous referee for very helpful remarks and
suggestions.
\end{acknowledgment}

\section{$G_{2}$-structures and torsion}

\setcounter{equation}{0}\label{secg2struct}The 14-dimensional group $G_{2}$
is the smallest of the five exceptional Lie groups and is closely related to
the octonions. In particular, $G_{2}$ can be defined as the automorphism
group of the octonion algebra. Taking the imaginary part of octonion
multiplication of the imaginary octonions defines a vector cross product on $%
V=\mathbb{R}^{7}$ and the group that preserves the vector cross product is
precisely $G_{2}$. A more detailed account of the relationship between
octonions and $G_{2}$ can be found in \cite{BaezOcto, GrigorianG2Review}.The
structure constants of the vector cross product define a $3$-form on $%
\mathbb{R}^{7}$, hence $G_{2}$ can alternatively be defined as the subgroup
of $GL\left( 7,\mathbb{R}\right) $ that preserves a particular $3$-form \cite%
{Joycebook}. In general, given an $n$-dimensional manifold $M$, a $G$%
-structure on $M$ for some Lie subgroup $G$ of $GL\left( n,\mathbb{R}\right) 
$ is a reduction of the frame bundle $F$ over $M$ to a principal subbundle $%
P $ with fibre $G$. A $G_{2}$-structure is then a reduction of the frame
bundle on a $7$-dimensional manifold $M$ to a $G_{2}$ principal subbundle.
It turns out that there is a $1$-$1$ correspondence between $G_{2}$%
-structures on a $7$-manifold and smooth $3$-forms $\varphi $ for which the $%
7$-form-valued bilinear form $B_{\varphi }$ as defined by (\ref{Bphi}) is
positive definite (for more details, see \cite{Bryant-1987} and the arXiv
version of \cite{Hitchin:2000jd}). 
\begin{equation}
B_{\varphi }\left( u,v\right) =\frac{1}{6}\left( u\lrcorner \varphi \right)
\wedge \left( v\lrcorner \varphi \right) \wedge \varphi  \label{Bphi}
\end{equation}%
Here the symbol $\lrcorner $ denotes contraction of a vector with the
differential form: 
\begin{equation*}
\left( u\lrcorner \varphi \right) _{mn}=u^{a}\varphi _{amn}.
\end{equation*}%
Note that we will also use this symbol for contractions of differential
forms using the metric.

A smooth $3$-form $\varphi $ is said to be \emph{positive }if $B_{\varphi }$
is the tensor product of a positive-definite bilinear form and a
nowhere-vanishing $7$-form. In this case, it defines a unique metric $%
g_{\varphi }$ and volume form $\mathrm{vol}$ such that for vectors $u$ and $%
v $, the following holds 
\begin{equation}
g_{\varphi }\left( u,v\right) \mathrm{vol}=\frac{1}{6}\left( u\lrcorner
\varphi \right) \wedge \left( v\lrcorner \varphi \right) \wedge \varphi
\label{gphi}
\end{equation}

In components we can rewrite this as 
\begin{equation}
\left( g_{\varphi }\right) _{ab}=\left( \det s\right) ^{-\frac{1}{9}}s_{ab}\ 
\text{where \ }s_{ab}=\frac{1}{144}\varphi _{amn}\varphi _{bpq}\varphi _{rst}%
\hat{\varepsilon}^{mnpqrst}.  \label{metricdefdirect}
\end{equation}%
Here $\hat{\varepsilon}^{mnpqrst}$ is the alternating symbol with $\hat{%
\varepsilon}^{12...7}=+1$. Following Joyce (\cite{Joycebook}), we will adopt
the following definition

\begin{definition}
Let $M$ be an oriented $7$-manifold. The pair $\left( \varphi ,g\right) $
for a positive $3$-form $\varphi $ and corresponding metric $g$ defined by (%
\ref{gphi}) will be referred to as a $G_{2}$-structure.
\end{definition}

Since a $G_{2}$-structure defines a metric and an orientation, it also
defines a Hodge star. Thus we can construct another $G_{2}$-invariant object
- the $4$-form $\ast \varphi $. Since the Hodge star is defined by the
metric, which in turn is defined by $\varphi $, the $4$-form $\ast \varphi $
depends non-linearly on $\varphi $. For convenience we will usually denote $%
\ast \varphi $ by $\psi $. We can also write down various contraction
identities for a $G_{2}$-structure $\left( \varphi ,g\right) $ and its
corresponding $4$-form $\psi $ \cite%
{bryant-2003,GrigorianYau1,karigiannis-2007}.

\begin{proposition}
\label{propcontractions}The $3$-form $\varphi $ and the corresponding $4$%
-form $\psi $ satisfy the following identities: 
\begin{subequations}%
\label{contids} 
\begin{eqnarray}
\varphi _{abc}\varphi _{mn}^{\ \ \ \ c} &=&g_{am}g_{bn}-g_{an}g_{bm}+\psi
_{abmn}  \label{phiphi1} \\
\varphi _{abc}\psi _{mnp}^{\ \ \ \ \ \ c} &=&3\left( g_{a[m}\varphi
_{np]b}-g_{b[m}\varphi _{np]a}\right)  \label{phipsi} \\
\psi _{abcd}\psi ^{mnpq} &=&24\delta _{a}^{[m}\delta _{b}^{n}\delta
_{c}^{p}\delta _{d}^{q]}+72\psi _{\lbrack ab}^{\ \ [mn}\delta _{c}^{p}\delta
_{d]}^{q]}-16\varphi _{\lbrack abc}^{{}}\varphi _{{}}^{[mnp}\delta _{d]}^{q]}
\label{psipsi0}
\end{eqnarray}%
\end{subequations}%
where $\left[ m\ n\ p\right] $ denotes antisymmetrization of indices and $%
\delta _{a}^{b}$ is the Kronecker delta, with $\delta _{b}^{a}=1$ if $a=b$
and $0$ otherwise.
\end{proposition}

The above identities can be of course further contracted - the details can
be found in \cite{GrigorianYau1,karigiannis-2007}. These identities and
their contractions are crucial whenever any tensorial calculations involving 
$\varphi $ and $\psi $ have to be done.

For a general $G$-structure, the spaces of $p$-forms decompose according to
irreducible representations of $G$. Given a $G_{2}$-structure, $2$-forms
split as $\Lambda ^{2}=\Lambda _{7}^{2}\oplus \Lambda _{14}^{2}$, where 
\begin{equation*}
\Lambda _{7}^{2}=\left\{ \alpha \lrcorner \varphi \text{: for a vector field 
}\alpha \right\}
\end{equation*}%
and 
\begin{equation*}
\Lambda _{14}^{2}=\left\{ \omega \in \Lambda ^{2}\text{: }\left( \omega
_{ab}\right) \in \mathfrak{g}_{2}\right\} =\left\{ \omega \in \Lambda ^{2}%
\text{: }\omega \lrcorner \varphi =0\right\} .
\end{equation*}%
The $3$-forms split as $\Lambda ^{3}=\Lambda _{1}^{3}\oplus \Lambda
_{7}^{3}\oplus \Lambda _{27}^{3}$, where the one-dimensional component
consists of forms proportional to $\varphi $, forms in the $7$-dimensional
component are defined by a vector field $\Lambda _{7}^{3}=\left\{ \alpha
\lrcorner \psi \text{: for a vector field }\alpha \right\} $, and forms in
the $27$-dimensional component are defined by traceless, symmetric matrices: 
\begin{equation}
\Lambda _{27}^{3}=\left\{ \chi \in \Lambda ^{3}:\chi _{abc}=\mathrm{i}%
_{\varphi }\left( h\right) =h_{[a}^{d}\varphi _{bc]d}^{{}}\text{ for }h_{ab}~%
\text{traceless, symmetric}\right\} .  \label{lam327}
\end{equation}%
By Hodge duality, similar decompositions exist for $\Lambda ^{4}$ and $%
\Lambda ^{5}$. In particular, we can define the $\Lambda _{27}^{4}$
component as 
\begin{equation}
\Lambda _{27}^{4}=\left\{ \chi \in \Lambda ^{4}:\chi _{abcd}=\ast \mathrm{i}%
_{\varphi }\left( h\right) =-\frac{4}{3}h_{[a}^{e}\psi _{\left\vert
e\right\vert bcd]}^{{}}\text{ for }h_{ab}~\text{traceless, symmetric}%
\right\} .  \label{lam427}
\end{equation}%
A detailed description of these representations is given in \cite%
{Bryant-1987,bryant-2003}. Also, formulae for projections of differential
forms onto the various components are derived in detail in \cite%
{GrigorianG2Torsion1, GrigorianYau1, karigiannis-2007}. Note that it is
sometimes convenient to consider the $1$ and $27$-dimensional components
together - then in (\ref{lam327}) and (\ref{lam427}) we simply drop the
condition for $h$ to be traceless. The only difference is that for an
arbitrary symmetric $h$, 
\begin{equation}
\left( \ast \mathrm{i}_{\varphi }\left( h\right) \right) _{abcd}=-\frac{4}{3}%
h_{[a}^{e}\psi _{\left\vert e\right\vert bcd]}^{{}}+\frac{1}{3}\left( \func{%
Tr}h\right) \psi _{abcd}
\end{equation}%
Also define the operators $\pi _{1}$, $\pi _{7}$, $\pi _{14}$ and $\pi _{27}$
to be the projections of differential forms onto the corresponding
representations. Sometimes we will also use $\pi _{1\oplus 27}$ to denote
the projection of $3$-forms or $4$-forms into $\Lambda _{1}^{3}\oplus
\Lambda _{27}^{3}$ or $\Lambda _{1}^{4}\oplus \Lambda _{27}^{4}$
respectively. For convenience, when writing out projections of forms, we
will sometimes just give the vector that defines the $7$-dimensional
component, the function that defines the $1$-dimensional component or the
symmetric $2$-tensor that defines the $1\oplus 27$ component whenever there
is no ambiguity. For instance,%
\begin{equation}
\begin{tabular}{lll}
$\pi _{1}\left( f\varphi \right) =f$ & $\pi _{1}\left( f\psi \right) =f$ & 
\\ 
$\pi _{7}\left( X\lrcorner \varphi \right) ^{a}=X^{a}$ & $\pi _{7}\left(
X\lrcorner \psi \right) ^{a}=X^{a}$ & $\pi _{7}\left( X\wedge \varphi
\right) ^{a}=X^{a}$ \\ 
$\pi _{1\oplus 27}\left( \mathrm{i}_{\varphi }\left( h\right) \right)
_{ab}=h_{ab}$ & $\pi _{1\oplus 27}\left( \ast \mathrm{i}_{\varphi }\left(
h\right) \right) _{ab}=h_{ab}$ & 
\end{tabular}
\label{projections}
\end{equation}

The \emph{intrinsic torsion }of a $G_{2}$-structure is defined by $\nabla
\varphi $, where $\nabla $ is the Levi-Civita connection for the metric $g$
that is defined by $\varphi $. Following \cite{karigiannis-2007}, it is easy
to see 
\begin{equation}
\nabla \varphi \in \Lambda _{7}^{1}\otimes \Lambda _{7}^{3}\cong W.
\label{torsphiW}
\end{equation}%
Here we define $W$ as the space $\Lambda _{7}^{1}\otimes \Lambda _{7}^{3}$.
Given (\ref{torsphiW}), we can write 
\begin{equation}
\nabla _{a}\varphi _{bcd}=T_{a}^{\ e}\psi _{ebcd}^{{}}  \label{codiffphi}
\end{equation}%
where $T_{ab}$ is the \emph{full torsion tensor}. Similarly, we can also
write 
\begin{equation}
\nabla _{a}\psi _{bcde}=-4T_{a[b}\varphi _{cde]}  \label{psitorsion}
\end{equation}%
We can also invert (\ref{codiffphi}) to get an explicit expression for $T$ 
\begin{equation}
T_{a}^{\ m}=\frac{1}{24}\left( \nabla _{a}\varphi _{bcd}\right) \psi ^{mbcd}.
\end{equation}%
This $2$-tensor fully defines $\nabla \varphi $ since pointwise, it has 49
components and the space $W$ is also 49-dimensional (pointwise). In general
we can split $T_{ab}$ according to representations of $G_{2}$ into \emph{%
torsion components}: 
\begin{equation}
T=\tau _{1}g+\tau _{7}\lrcorner \varphi +\tau _{14}+\tau _{27}
\end{equation}%
where $\tau _{1}$ is a function, and gives the $\mathbf{1}$ component of $T$%
. We also have $\tau _{7}$, which is a $1$-form and hence gives the $\mathbf{%
7}$ component, and, $\tau _{14}\in \Lambda _{14}^{2}$ gives the $\mathbf{14}$
component and $\tau _{27}$ is traceless symmetric, giving the $\mathbf{27}$
component. Hence we can split $W$ as 
\begin{equation}
W=W_{1}\oplus W_{7}\oplus W_{14}\oplus W_{27}.
\end{equation}%
As it was originally shown by Fern\'{a}ndez and Gray \cite{FernandezGray},
there are in fact a total of 16 torsion classes of $G_{2}$-structures that
arise as the $G_{2}$-invariant subspaces of $W$ to which $\nabla \varphi $
belongs. Moreover, as shown in \cite{karigiannis-2007}, the torsion
components $\tau _{i}$ relate directly to the expression for $d\varphi $ and 
$d\psi $. In fact, in our notation, 
\begin{subequations}%
\label{dptors} 
\begin{eqnarray}
d\varphi &=&4\tau _{1}\psi -3\tau _{7}\wedge \varphi -3\ast \mathrm{i}%
_{\varphi }\left( \tau _{27}\right)  \label{dphi} \\
d\psi &=&-4\tau _{7}\wedge \psi -2\ast \tau _{14}.  \label{dpsi}
\end{eqnarray}%
\end{subequations}%
Note that in the literature (\cite{bryant-2003,CleytonIvanovConf}, for
example) a slightly different convention for torsion components is sometimes
used. Our $\tau _{1}$ then corresponds to $\frac{1}{4}\tau _{0}$, $\tau _{7}$
corresponds to $-\tau _{1}$ in their notation,\ $\mathrm{i}_{\varphi }\left(
\tau _{27}\right) $ corresponds to $-\frac{1}{3}\tau _{3}$ and $\tau _{14}$
corresponds to $\frac{1}{2}\tau _{2}$. Similarly, our torsion classes $%
W_{1}\oplus W_{7}\oplus W_{14}\oplus W_{27}$ correspond to $W_{0}\oplus
W_{1}\oplus W_{2}\oplus W_{3}$.

An important special case is when the $G_{2}$-structure is said to be
torsion-free, that is, $T=0$. This is equivalent to $\nabla \varphi =0$ and
also equivalent, by Fern\'{a}ndez and Gray, to $d\varphi =d\psi =0$.
Moreover, a $G_{2}$-structure is torsion-free if and only if the holonomy of
the corresponding metric is contained in $G_{2}$ \cite{Joycebook}. The
holonomy group is then precisely equal to $G_{2}$ if and only if the
fundamental group $\pi _{1}$ is finite.

If $d\varphi =0$, then we say $\varphi $ defines a \emph{closed }$G_{2}$%
-structure. Each of the torsion components in (\ref{dphi}) has to vanish
separately, so $\tau _{1}$, $\tau _{7}$ and $\tau _{27}$ are all zero, and
the only non-zero torsion component remaining is $\tau _{14}$. This class of 
$G_{2}$-structures has been important in the study of Laplacian flows of $%
\varphi $ \cite{bryant-2003,BryantXu,XuYe}. If instead, $d\psi =0$, then we
say that we have a \emph{co-closed }$G_{2}$-structure. In this case, $\tau
_{7}$ and $\tau _{14}$ vanish in (\ref{dpsi}) and we are left with $\tau
_{1} $ and $\tau _{27}$ components. In particular, the torsion tensor $%
T_{ab} $ is now symmetric. In this paper we will mostly be concerned with
co-closed $G_{2}$-structures.

We will also require a number of differential identities for differential
forms on manifolds with a co-closed $G_{2}$-structure. In many ways this is
the explicit version of Bryant's exterior derivative identities \cite%
{bryant-2003}, but these identities will be useful for us later on. The
projections are defined as in (\ref{projections}).

Throughout the paper we will be using the following notation. Given a $p$%
-form $\omega $, the rough Laplacian is defined by 
\begin{equation}
\nabla ^{2}\omega =g^{ab}\nabla _{a}\nabla _{b}\omega =-\nabla ^{\ast
}\nabla \omega .  \label{roughlap}
\end{equation}%
For a vector field $X$, define the \emph{divergence} of $X$ as 
\begin{equation}
\func{div}X=\nabla _{a}X^{a}  \label{divX}
\end{equation}%
This operator can be extended to a symmetric $2$-tensor $h$ 
\begin{equation}
\left( \func{div}h\right) _{b}=\nabla ^{a}h_{ab}  \label{divh}
\end{equation}%
Also, for a vector $X$, we can use the $G_{2}$-structure $3$-form $\varphi $
to define a \textquotedblleft curl\textquotedblright\ operator, similar to
the standard one on $\mathbb{R}^{3}$: 
\begin{equation}
\left( \func{curl}X\right) ^{a}=\left( \nabla _{b}X_{c}\right) \varphi ^{abc}
\label{curlX}
\end{equation}%
This curl operator can then also be extended to $2$-tensor. Given a $2$%
-tensor $\beta _{ab}$, 
\begin{equation}
\left( \func{curl}\beta \right) _{ab}=\left( \nabla _{m}^{{}}\beta
_{an}^{{}}\right) \varphi _{b}^{\ mn}  \label{curlw}
\end{equation}%
From the context it will be clear whether the curl operator is applied to a
vector or a $2$-tensor. Note that when $\beta _{ab}$ is symmetric, $\func{%
curl}w$ is traceless. As in \cite{GrigorianPhD}, we can also use the $G_{2}$%
-structure $3$-form to define a product $\alpha \circ \beta $ of two $2$%
-tensors $\alpha $ and $\beta $%
\begin{equation}
\left( \alpha \circ \beta \right) _{ab}=\varphi _{amn}\varphi _{bpq}\alpha
^{mp}\beta ^{nq}  \label{phphprod}
\end{equation}%
While the product in (\ref{phphprod}) can be defined for any $2$-tensors on
a manifold with a $G_{2}$-structure, for us it will be most useful when
restricted to symmetric tensors. It is then easy to see that (\ref{phphprod}%
) defines a commutative product on the space of symmetric $2$-tensors. It is
however non-associative, and so defines a non-trivial non-associative
algebra on symmetric $2$-tensors. For convenience, we will define a standard
inner product on symmetric $2$-tensors%
\begin{equation}
\left\langle \alpha ,\beta \right\rangle =\alpha _{ab}\beta _{mn}g^{am}g^{bn}
\label{symip}
\end{equation}

\begin{proposition}
\label{PropComponents}Suppose we have a co-closed $G_{2}$-structure on a
manifold $M$ with $3$-form $\varphi $ and dual $4$-form $\psi $. Let $\chi
\in \Lambda ^{3}$ be given by 
\begin{equation}
\chi =X\lrcorner \psi +3\mathrm{i}_{\varphi }\left( h\right)
\end{equation}%
then, the co-differential $d^{\ast }\chi $ is given by 
\begin{subequations}%
\label{d3scomps}%
\begin{eqnarray}
\left( d^{\ast }\chi \right) _{bc} &=&-\left( \left( \func{div}h\right)
\lrcorner \varphi \right) _{bc}-2\left( \func{curl}h\right) _{\left[ bc%
\right] }+\nabla _{m}^{{}}X_{n}^{{}}\psi _{\ \ \ bc}^{mn}  \label{d3scompall}
\\
&&-\left( \func{Tr}T\right) X_{{}}^{a}\varphi
_{abc}^{{}}+T_{mn}^{{}}X_{{}}^{n}\varphi _{\ \
bc}^{m}+2X_{m}T_{n[b}^{{}}\varphi _{\ \ \ c]}^{mn}-T_{\ \ n}^{m}h^{np}\psi
_{mnbc}  \notag \\
\left( \pi _{7}d^{\ast }\chi \right) _{a} &=&\frac{2}{3}\left( \left( \func{%
curl}X\right) _{a}-\left( \func{div}h\right) _{a}-\frac{1}{2}\nabla _{a}%
\func{Tr}h\right.  \label{d3scomp7} \\
&&\left. -X_{a}\func{Tr}T+T_{ab}X^{b}-\varphi _{abc}^{{}}T_{\
d}^{b}h_{{}}^{dc}\right)  \notag
\end{eqnarray}%
\end{subequations}%
The type decomposition of the exterior derivative $d\chi $ is 
\begin{subequations}%
\label{d3comps}%
\begin{eqnarray}
\pi _{1}d\chi &=&\frac{4}{7}\left( \func{div}X+\frac{1}{2}\func{Tr}T\func{Tr}%
h-\frac{1}{2}T_{ab}h^{ab}\right) \\
\left( \pi _{7}d\chi \right) _{a} &=&\frac{1}{2}\left( \nabla _{a}\func{Tr}%
h-\left( \func{div}h\right) _{a}-\left( \func{curl}X\right)
_{a}-2T_{ab}X^{b}\right) \\
\left( \pi _{1\oplus 27}d\chi \right) _{ab} &=&3\left( -\nabla
_{(a}X_{b)}+\left( \func{curl}h\right) _{\left( ab\right) }+\frac{1}{3}%
\left( \func{div}X\right) g_{ab}\right. \\
&&+\frac{1}{2}\left( T\circ h\right) _{ab}+T_{m(a}^{{}}h_{b)}^{\ \ m}-\frac{1%
}{2}\left( \func{Tr}h\right) T_{ab}  \notag \\
&&\left. -\frac{1}{2}\left( \func{Tr}T\right) h_{ab}-\frac{1}{6}g_{ab}\left( 
\func{Tr}T\right) \left( \func{Tr}h\right) +\frac{1}{6}\left\langle
h,T\right\rangle g_{ab}\right)  \notag
\end{eqnarray}%
\end{subequations}%
Also, up to the lower order terms involving the torsion, the $\Lambda
_{1}^{3}$ and $\Lambda _{7}^{3}$ components of $dd^{\ast }\chi $ are given by%
\begin{subequations}%
\label{ddscomps} 
\begin{eqnarray}
\pi _{1}dd^{\ast }\chi &=&-\frac{2}{7}\left( \func{div}\left( \func{div}%
h\right) +\frac{1}{2}\nabla ^{2}\func{Tr}h\right) +l.o.t. \\
\left( \pi _{7}dd^{\ast }\chi \right) _{a} &=&\frac{1}{2}\left( \nabla
_{a}\left( \func{div}X\right) -\nabla ^{2}X_{a}-\func{curl}\left( \func{div}%
h\right) _{a}\right) +l.o.t.
\end{eqnarray}%
\end{subequations}%
Similarly, up to the lower order terms, the $\Lambda _{1}^{3}$ and $\Lambda
_{7}^{3}$ components of $d^{\ast }d\chi $ are given by%
\begin{subequations}%
\label{dsdcomps}%
\begin{eqnarray}
\pi _{1}d^{\ast }d\chi &=&\frac{2}{7}\left( \func{div}\left( \func{div}%
h\right) -\nabla ^{2}\func{Tr}h\right) +l.o.t. \\
\left( \pi _{7}d^{\ast }d\chi \right) _{a} &=&\frac{1}{2}\left( \func{curl}%
\left( \func{div}h\right) _{a}-\nabla _{a}\left( \func{div}X\right) -\nabla
^{2}X_{a}\right) +l.o.t.
\end{eqnarray}%
\end{subequations}%
.
\end{proposition}

\begin{proof}
These identities are found just by manipulating $G_{2}$ representation
components using contraction identities between $\varphi $ and $\psi $. For
the second order identities in order to isolate the highest order terms we
note that $\pi _{7}\left( Riem\right) $ and the Ricci tensor are expressed
solely in terms of the full torsion tensor \cite{karigiannis-2007}.
\end{proof}

\section{Deformations of $\protect\psi $}

\label{SecDeformPsi}\setcounter{equation}{0}Usually deformations of a $G_{2}$%
-structure are done via deformations of the $3$-form $\varphi $, and from
that deformations of associated quantities - $g,$ $\psi $ and the torsion
are calculated. In particular, infinitesimal deformations of all these
quantities have been written down in \cite{karigiannis-2007}, while the
general non-infinitesimal expression were derived in \cite%
{GrigorianG2Torsion1}. Since we will be considering a flow of $\psi $, we
need to re-derive all the infinitesimal results from \cite{karigiannis-2007}
using a deformation of $\psi $ as a starting point. Let $\left( \varphi
,g\right) $ be a $G_{2}$-structure. Using the metric, define $\psi =\ast
\varphi $. Suppose $\chi \in \Lambda ^{3}$, then $\ast \chi \in \Lambda ^{4}$%
. Consider a deformation of the $G_{2}$-structure via a deformation of $\psi 
$: 
\begin{equation}
\psi \longrightarrow \tilde{\psi}=\psi +\ast \chi  \label{psit0}
\end{equation}%
Assuming $\tilde{\psi}$ remains a positive $4$-form, it defines a new $G_{2}$%
-structure $\left( \tilde{\varphi},\tilde{g}\right) $. The $3$-form $\tilde{%
\varphi}$ is given by:%
\begin{eqnarray}
\tilde{\varphi}_{abc} &=&\tilde{\ast}\left( \psi +\ast \chi \right) _{abc} 
\notag \\
&=&\frac{1}{3!}\frac{1}{\sqrt{\det \tilde{g}}}\hat{\varepsilon}%
^{mnpqrst}\left( \psi _{qrst}+\ast \chi _{qrst}\right) \tilde{g}_{ma}\tilde{g%
}_{nb}\tilde{g}_{pc}  \notag \\
&=&\left( \frac{\det g}{\det \tilde{g}}\right) ^{\frac{1}{2}}\left( \varphi
^{mnp}+\chi ^{mnp}\right) \tilde{g}_{ma}\tilde{g}_{nb}\tilde{g}_{pc}
\label{phit0}
\end{eqnarray}%
In particular, 
\begin{equation}
\tilde{\varphi}^{\tilde{a}\tilde{b}\tilde{c}}=\left( \frac{\det g}{\det 
\tilde{g}}\right) ^{\frac{1}{2}}\left( \varphi ^{mnp}+\chi ^{mnp}\right)
\label{phit3}
\end{equation}%
where the raised indices with tildes are raised with the new inverse metric $%
\tilde{g}^{-1}$. The new metric can be found via the following identity. For
any $G_{2}$-structure $\left( \varphi ,g\right) $, from (\ref{contids}) we
find that 
\begin{equation*}
\psi _{amnp}\psi _{bqrs}\varphi ^{mnq}\varphi ^{prs}=16\varphi _{ap}^{\ \
q}\varphi _{bq}^{\ \ p}=-96g_{ab}
\end{equation*}%
Therefore for the $G_{2}$-structure $\left( \tilde{\varphi},\tilde{g}\right) 
$, we have 
\begin{eqnarray}
\tilde{g}_{ab} &=&-\frac{1}{96}\tilde{\psi}_{amnp}\tilde{\psi}_{bqrs}\tilde{%
\varphi}^{mnq}\tilde{\varphi}^{prs}  \notag \\
&=&-\frac{1}{96}\left( \frac{\det g}{\det \tilde{g}}\right) \left( \psi
_{amnp}+\ast \chi _{amnp}\right) \left( \psi _{bqrs}+\ast \chi
_{bqrs}\right) \left( \varphi ^{mnq}+\chi ^{mnq}\right) \left( \varphi
^{prs}+\chi ^{prs}\right)  \label{gtl}
\end{eqnarray}%
Without the determinant factor, this is a 4th order expression in $\chi $.
We can also obtain an expression for the inverse metric. Again, from (\ref%
{contids}), we have 
\begin{equation*}
\varphi ^{amn}\varphi ^{bpq}\psi _{mnpq}=4\varphi ^{amn}\varphi _{\ \
mn}^{b}=24g^{ab}
\end{equation*}%
Therefore for $\left( \tilde{\varphi},\tilde{g}\right) $, we have 
\begin{eqnarray}
\tilde{g}^{\tilde{a}\tilde{b}} &=&\frac{1}{24}\tilde{\varphi}^{\tilde{a}%
\tilde{m}\tilde{n}}\tilde{\varphi}^{\tilde{b}\tilde{p}\tilde{q}}\tilde{\psi}%
_{mnpq}  \notag \\
&=&\frac{1}{24}\left( \frac{\det g}{\det \tilde{g}}\right) \left( \varphi
^{amn}+\chi ^{amn}\right) \left( \varphi ^{bpq}+\chi ^{bpq}\right) \left(
\psi _{mnpq}+\ast \chi _{mnpq}\right)  \label{gtr}
\end{eqnarray}%
Suppose now $\chi $ is given by 
\begin{equation}
\chi =X\lrcorner \psi +3\mathrm{i}_{\varphi }\left( h\right)
\label{chidecomp}
\end{equation}%
where $v$ is vector and $h$ is a symmetric $2$-tensor. Also suppose that we
have a one-parameter family $\psi \left( t\right) $ given by 
\begin{eqnarray}
\frac{d}{dt}\psi &=&\ast \chi  \label{genpsiflow} \\
&=&-X\wedge \varphi +3\ast \mathrm{i}_{\varphi }\left( h\right)
\end{eqnarray}%
Then the evolution of $3$-form, metric, the inverse metric and the volume
form is given in the following Proposition.

\begin{proposition}
\label{propflow}Under the flow (\ref{genpsiflow}), the evolution of related
objects is given by:%
\begin{subequations}%
\begin{eqnarray}
\frac{d}{dt}\sqrt{\det g} &=&\frac{3}{4}\left( \func{Tr}h\right) \sqrt{\det g%
}  \label{dtdetg} \\
\frac{d}{dt}g_{ab} &=&\frac{1}{2}\left( \func{Tr}h\right) g_{ab}-2h_{ab}
\label{dtg} \\
\frac{d}{dt}g^{ab} &=&-\frac{1}{2}\left( \func{Tr}h\right) g^{ab}+2h^{ab}
\label{dtginv} \\
\frac{d}{dt}\varphi &=&\frac{3}{4}\pi _{1}\chi +\pi _{7}\chi -\pi _{27}\chi
\label{dtphi} \\
&=&\frac{3}{4}\left( \func{Tr}h\right) \varphi +v\lrcorner \psi -3\mathrm{i}%
_{\varphi }\left( h\right)  \notag
\end{eqnarray}%
\end{subequations}%
\end{proposition}

\begin{proof}
From (\ref{gtl}) we find that under the flow (\ref{genpsiflow}), $g\left(
t\right) $ is given by 
\begin{eqnarray}
g_{ab}\left( t\right) &=&-\frac{1}{96}\left( \frac{\det g}{\det \left(
g\left( t\right) \right) }\right) \left( \psi _{amnp}+t\ast \chi
_{amnp}\right) \left( \psi _{bqrs}+t\ast \chi _{bqrs}\right)  \label{gt1} \\
&&\times \left( \varphi ^{mnq}+t\chi ^{mnq}\right) \left( \varphi
^{prs}+t\chi ^{prs}\right) +O\left( t^{2}\right)  \notag
\end{eqnarray}%
where $g\left( 0\right) =g$. Now let us find$\left. \frac{d}{dt}\right\vert
_{t=0}g\left( t\right) $:%
\begin{eqnarray}
\left. \frac{d}{dt}\right\vert _{t=0}g\left( t\right) _{ab} &=&\left( \left. 
\frac{d}{dt}\right\vert _{t=0}\left( \det g\left( t\right) \right)
^{-1}\right) \left( -\frac{\det g}{96}\right) \psi _{amnp}\psi
_{bqrs}\varphi ^{mnq}\varphi ^{prs}  \notag \\
&&-\frac{1}{96}\ast \chi _{amnp}\psi _{bqrs}\varphi ^{mnq}\varphi ^{prs}-%
\frac{1}{96}\ast \chi _{bqrs}\psi _{amnp}\varphi ^{mnq}\varphi ^{prs}  \notag
\\
&&-\frac{1}{96}\psi _{amnp}\psi _{bqrs}\chi ^{mnq}\varphi ^{prs}-\frac{1}{96}%
\psi _{amnp}\psi _{bqrs}\varphi ^{mnq}\chi ^{prs}  \notag \\
&=&-g_{ab}\left( \det g\right) ^{-1}\left( \left. \frac{d}{dt}\right\vert
_{t=0}\det g\left( t\right) \right) -\frac{1}{24}\chi _{amnp}^{{}}\varphi
_{{}}^{mnq}\varphi _{bq}^{\ \ p}  \label{ddtgt} \\
&&-\frac{1}{24}\ast \chi _{bqrs}^{{}}\varphi _{ap}^{\ \ q}\varphi
_{{}}^{prs}-\frac{1}{24}\psi _{amnp}^{{}}\varphi _{bq}^{\ \ p}\chi ^{mnq} 
\notag \\
&&-\frac{1}{24}\varphi _{ap}^{\ \ q}\psi _{bqrs}\chi ^{prs}  \notag
\end{eqnarray}%
where we have used a contracted version of the identity (\ref{phipsi}). Now
if we substitute (\ref{chidecomp}) into (\ref{ddtgt}), and simplify further
using the identities (\ref{contids}), we will find that%
\begin{subequations}%
\label{delg} 
\begin{equation}
\left. \frac{d}{dt}\right\vert _{t=0}g\left( t\right) _{ab}=-\left( \det
g\right) ^{-1}\left( \left. \frac{d}{dt}\right\vert _{t=0}\det g\left(
t\right) \right) g_{ab}+2\left( \func{Tr}h\right) g_{ab}-2h_{ab}
\end{equation}%
and similarly, for the inverse metric, 
\begin{equation}
\left. \frac{d}{dt}\right\vert _{t=0}g\left( t\right) ^{ab}=-\left( \det
g\right) ^{-1}\left( \left. \frac{d}{dt}\right\vert _{t=0}\det g\left(
t\right) \right) g^{ab}+\left( \func{Tr}h\right) g^{ab}+2h^{ab}
\end{equation}%
\end{subequations}%
. Now however, using the fact that the derivative of a determinant gives the
trace, we get that 
\begin{eqnarray*}
\left( \left. \frac{d}{dt}\right\vert _{t=0}\det g\left( t\right) \right)
&=&\det g\left( g^{ab}\left. \frac{d}{dt}\right\vert _{t=0}g\left( t\right)
_{ab}\right) \\
&=&-7\left( \left. \frac{d}{dt}\right\vert _{t=0}\det g\left( t\right)
\right) +12\det g\func{Tr}h
\end{eqnarray*}%
Hence, 
\begin{equation}
\left( \left. \frac{d}{dt}\right\vert _{t=0}\det g\left( t\right) \right) =%
\frac{3}{2}\det g\func{Tr}h  \label{delgam}
\end{equation}%
Therefore, substituting this into (\ref{delg}) we obtain (\ref{dtg}) and (%
\ref{dtginv}). Also, from (\ref{delgam}) we immediately obtain the
expression for $\frac{d}{dt}\sqrt{\det g}$ (\ref{dtdetg}). Now to compute $%
\frac{d}{dt}\varphi $, we first find from (\ref{phit0}) that under the flow (%
\ref{genpsiflow})%
\begin{equation*}
\varphi \left( t\right) _{abc}=\left( \frac{\det g}{\det g\left( t\right) }%
\right) ^{\frac{1}{2}}\left( \varphi ^{mnp}+t\chi ^{mnp}\right) g\left(
t\right) _{ma}g\left( t\right) _{nb}g\left( t\right) _{pc}+O\left(
t^{2}\right)
\end{equation*}%
Hence, 
\begin{eqnarray}
\left. \frac{d}{dt}\right\vert _{t=0}\varphi \left( t\right) _{abc} &=&-%
\frac{1}{2}\left( \det g\right) ^{-1}\left( \left. \frac{d}{dt}\right\vert
_{t=0}\det g\left( t\right) \right) \varphi _{abc}  \notag \\
&&+\chi _{abc}+3\left. \frac{d}{dt}\right\vert _{t=0}g\left( t\right)
_{m[a}^{{}}\varphi _{\ bc]}^{m}  \notag \\
&=&\frac{3}{4}\left( \func{Tr}h\right) \varphi _{abc}+\chi _{abc}+3\mathrm{i}%
_{\varphi }\left( \frac{1}{2}\left( \func{Tr}h\right) g-2h\right) _{abc}
\end{eqnarray}%
This can be rewritten as 
\begin{eqnarray*}
\left. \frac{d}{dt}\right\vert _{t=0}\varphi \left( t\right) &=&\frac{3}{4}%
\left( \func{Tr}h\right) \varphi +X\lrcorner \psi -3\mathrm{i}_{\varphi
}\left( h\right) \\
&=&\frac{3}{4}\pi _{1}\chi +\pi _{7}\chi -\pi _{27}\chi
\end{eqnarray*}%
and thus we get (\ref{dtphi}).
\end{proof}

Now that we know how $\varphi $ evolves, we can easily work out the
evolution of the torsion tensor $T_{ab}$.

\begin{proposition}
\label{PropDtdtgen}The evolution of the torsion tensor $T_{ab}$ under the
flow (\ref{genpsiflow}) is given by%
\begin{equation}
\frac{dT_{ab}}{dt}=\frac{1}{4}\left( \func{Tr}h\right) T_{ab}-T_{a}^{\
c}h_{cb}^{{}}-T_{a}^{\ c}X_{{}}^{d}\varphi _{dcb}^{{}}+\left( \func{curl}%
h\right) _{ab}+\nabla _{a}X_{b}-\frac{1}{4}\left( \nabla _{c}^{{}}\func{Tr}%
h\right) \varphi _{\ ab}^{c}  \label{dttorspsi}
\end{equation}
\end{proposition}

\begin{proof}
From \cite{karigiannis-2007} we infer that under a general evolution of $%
\varphi $ given by 
\begin{equation*}
\frac{d}{dt}\varphi =X\lrcorner \psi +3\mathrm{i}_{\varphi }\left( s\right)
\end{equation*}%
for some symmetric $2$-tensor $s$, the evolution of the torsion tensor $%
T_{ab}$ is given by 
\begin{eqnarray}
\frac{dT_{ab}}{dt} &=&T_{a}^{\ \ c}s_{cb}^{{}}-T_{a}^{\ \ c}X^{d}\varphi
_{dcb}-\left( \nabla _{c}^{{}}s_{da}^{{}}\right) \varphi _{\ \
b}^{cd}+\nabla _{a}X_{b}  \notag \\
&=&T_{a}^{\ \ c}s_{cb}^{{}}-T_{a}^{\ \ c}X^{d}\varphi _{dcb}-\left( \func{%
curl}s\right) _{ab}+\nabla _{a}X_{b}  \label{dttorsphi}
\end{eqnarray}%
Note that compared with \cite{karigiannis-2007}, some of the signs are
different due to a different sign convention for $\psi $, which also leads
to a different sign for $T$ and $X$. From (\ref{dtphi}) we have that in our
case, 
\begin{equation*}
s=\frac{1}{4}\left( \func{Tr}h\right) g-h
\end{equation*}%
Therefore, substituting this into (\ref{dttorsphi}) we get (\ref{dttorspsi}).
\end{proof}

A key motivation for studying flows of closed $G_{2}$-structures within a
fixed cohomology class of the $3$-form $\varphi $ was that a critical point
of the volume functional (\ref{volfunc}) within the cohomology class $\left[
\varphi \right] $ corresponds a torsion-free $G_{2}$-structure. It is
trivial to adapt Hitchin's proof of this fact \cite{Hitchin:2000jd-arxiv} to
co-closed $G_{2}$-structures.

\begin{proposition}
\label{propvolextr}Let $M$ be a compact $7$-manifold, and suppose the $4$%
-form $\psi $ defines a co-closed $G_{2}$-structure on $M$, so that $d\psi
=0 $. Let $\varphi =\ast \psi $ be the corresponding $3$-form. Define the
volume functional $V$ 
\begin{equation}
V\,\left( \psi \right) =\frac{1}{7}\int_{M}\varphi \wedge \psi .
\label{volfuncpsi}
\end{equation}%
Then $\psi $ defines a torsion-free $G_{2}$-structure if and only if it is a
critical point of the functional $V$ restricted to the cohomology class $%
[\psi ]\in H^{4}\left( M,\mathbb{R}\right) $.
\end{proposition}

\begin{proof}
Consider the variation of $V$:%
\begin{eqnarray}
\delta V\left( \dot{\psi}\right) &=&\frac{1}{7}\int_{M}\dot{\varphi}\wedge
\psi +\frac{1}{7}\int_{M}\varphi \wedge \dot{\psi}  \notag \\
&=&\frac{1}{7}\int_{M}\left( \frac{3}{4}\pi _{1}\ast \dot{\psi}+\pi _{7}\ast 
\dot{\psi}-\pi _{27}\ast \dot{\psi}\right) \wedge \psi  \notag \\
&&+\frac{1}{7}\int_{M}\varphi \wedge \dot{\psi}  \notag \\
&=&\frac{1}{4}\int_{M}\varphi \wedge \dot{\psi}
\end{eqnarray}%
where we have used (\ref{dtphi}). Now suppose $\dot{\psi}=d\eta $ for some $%
3 $-form $\eta $, so that we vary in the same cohomology class. Now%
\begin{equation}
\delta V\left( \dot{\psi}\right) =\frac{1}{4}\int_{M}\varphi \wedge d\eta =%
\frac{1}{4}\int d\varphi \wedge \eta  \label{delV}
\end{equation}%
Thus $\delta V=0$ for all $\eta $ if and only if 
\begin{equation*}
d\varphi =0\text{.}
\end{equation*}%
Since we already have $d\psi =0$, this is satisfied if and only if the $%
G_{2} $-structure is torsion-free.
\end{proof}

In the arXiv version of \cite{Hitchin:2000jd}, Hitchin has shown that
critical points of the functional $V\left( \varphi \right) $ on $3$-forms
are non-degenerate in the directions transverse to the action of the
diffeomorphism group $Diff\left( M\right) $. Here we adapt the proof from 
\cite{Hitchin:2000jd} to show an analogous result for the functional $%
V\left( \psi \right) $ on $4$-forms.

\begin{proposition}
Suppose the $4$-form $\psi $ defines a torsion-free $G_{2}$-structure on a
compact $7$-manifold $M$. Then the Hessian of the functional $V$ (\ref%
{volfuncpsi}) at $\psi $ is non-degenerate transverse to the action of $%
Diff\left( M\right) $.
\end{proposition}

\begin{proof}
Since $\psi $ defines a torsion-free $G_{2}$-structure, it is a critical
point of the functional $V$. Let us consider infinitesimal deformations of $%
\psi $ by an exact form $d\chi $, for $\chi \in \Lambda ^{3}$. Then, from (%
\ref{dtphi}), the deformation of the dual $3$-form $\varphi $ is given by%
\begin{equation}
D\left( d\chi \right) =\frac{3}{4}\ast \pi _{1}d\chi +\ast \pi _{7}d\chi
-\ast \pi _{27}d\chi .  \label{Ddchi}
\end{equation}

Now, the tangent space of the orbit of $Diff\left( M\right) $ consists of
the forms $\mathcal{L}_{X}\psi $ where $X$ is some vector field. Suppose $%
d\chi =\mathcal{L}_{X}\psi =d\left( X\lrcorner \psi \right) $ for some
vector field $X$. Then by diffeomorphism invariance, we get that $D\left(
d\chi \right) =\mathcal{L}_{X}\varphi =d\left( X\lrcorner \varphi \right) $
and hence $\varphi $ remains closed along the orbits of $Diff\left( M\right) 
$. So now suppose that $\psi $ is a critical point but which is degenerate
in some direction $d\chi $ which is transverse to orbits of $Diff\left(
M\right) $, that is, $\psi $ remains torsion-free in the direction of $d\chi 
$. Hence for proof by contradiction, we have to show now that if $d\chi $
orthogonal to $d\left( X\lrcorner \varphi \right) $ for any vector field $X$
and $d\left( D\left( d\chi \right) \right) =0$, then $d\chi =0$.

If $d\chi $ is orthogonal to orbits of $Diff\left( M\right) $ then for any
vector field $X$, 
\begin{eqnarray}
0 &=&\int_{M}\left\langle d\chi ,\mathcal{L}_{X}\psi \right\rangle \mathrm{%
vol}=\int_{M}\left\langle d\chi ,d\left( X\lrcorner \psi \right)
\right\rangle \mathrm{vol}  \notag \\
&=&\int_{M}\left\langle d^{\ast }d\chi ,X\lrcorner \psi \right\rangle 
\mathrm{vol}  \label{pi7dsd0}
\end{eqnarray}%
where $\left\langle \cdot ,\cdot \right\rangle $ is the standard inner
product with respect to the metric $g$ defined by the $G_{2}$-structure on $%
M $. Since (\ref{pi7dsd0}) is true for any $X$, this means that 
\begin{equation}
\pi _{7}d^{\ast }d\chi =0.  \label{pi7dsd1}
\end{equation}

Now from the Hodge Theorem, we can write $\chi $ as 
\begin{equation}
\chi =H\left( \chi \right) +dGd^{\ast }\chi +d^{\ast }Gd\chi
\end{equation}%
where $H\left( \chi \right) $ gives the harmonic part and $G$ is the Green's
operator for the Hodge Laplacian. An important property of $G$ is that it
commutes with $d$, $d^{\ast }$ and the projections onto representation
components. Then we have 
\begin{equation*}
d\chi =dd^{\ast }Gd\chi =dGd^{\ast }d\chi
\end{equation*}%
However, from (\ref{pi7dsd1}), 
\begin{equation*}
\pi _{7}Gd^{\ast }d\chi =0.
\end{equation*}%
So, we can say that $d\chi =d\eta $ for $\eta =Gd^{\ast }d\chi \in \Lambda
_{1}^{3}\oplus \Lambda _{27}^{3},$ and moreover~%
\begin{equation}
d^{\ast }\eta =0
\end{equation}%
hence we can assume that $\chi \in \Lambda _{1}^{3}\oplus \Lambda _{27}^{3}$
and $d^{\ast }\chi =0$. In particular, $\pi _{7}d^{\ast }\chi =0$. Then if
we suppose $\chi $ is given by $\chi =3\mathrm{i}_{\varphi }\left( h\right) $%
, from Proposition \ref{PropComponents} we get that 
\begin{equation*}
\left( \func{div}h\right) _{a}=-\frac{1}{2}\nabla _{a}\func{Tr}h
\end{equation*}%
and therefore, from (\ref{d3comps}), 
\begin{eqnarray}
\pi _{7}d\chi &=&\frac{3}{4}d\func{Tr}\left( h\right) \wedge \varphi  \notag
\\
&=&d\left( \frac{3}{4}\left( \func{Tr}h\right) \varphi \right)
\end{eqnarray}%
Moreover, since $\chi \in \Lambda _{1}^{3}\oplus \Lambda _{27}^{3}$, from (%
\ref{d3comps}) we have that 
\begin{equation*}
\pi _{1}d\chi =0.
\end{equation*}%
Thus from the condition $d\left( D\left( d\chi \right) \right) =0$ we have 
\begin{equation}
d^{\ast }d\chi -2d^{\ast }\pi _{7}d\chi =0  \label{dsdchipi7}
\end{equation}%
Now note that 
\begin{eqnarray*}
\Delta \pi _{7}d\chi &=&dd^{\ast }\pi _{7}d\chi +d^{\ast }d\pi _{7}d\chi \\
&=&dd^{\ast }\pi _{7}d\chi
\end{eqnarray*}%
since $\pi _{7}d\chi $ is exact. Hence by applying the exterior derivative
to (\ref{dsdchipi7}), we find that 
\begin{equation*}
\Delta d\chi -2\Delta \pi _{7}d\chi =0
\end{equation*}%
However the Hodge Laplacian of a torsion-free $G_{2}$-structure commutes
with the component projections and we recover%
\begin{equation*}
\pi _{27}\Delta d\chi -\pi _{7}\Delta d\chi =0.
\end{equation*}%
Each of the components must vanish individually, and so 
\begin{equation*}
\Delta d\chi =0
\end{equation*}%
and thus $d\chi =0$ as required.
\end{proof}

\section{Laplacian of $\protect\psi $}

\label{SecLapPsi}\setcounter{equation}{0}Let us now look at the properties
of $\Delta \psi $. For now consider a generic $G_{2}$-structure, so that 
\begin{equation*}
\Delta \psi =dd^{\ast }\psi +d^{\ast }d\psi
\end{equation*}%
Note that since $\psi =\ast \varphi $, we have 
\begin{equation*}
\Delta \psi =\ast \Delta \varphi
\end{equation*}%
so in particular, for the type decomposition of $\Delta \psi $ it is enough
to look at $\Delta \varphi $.

\begin{proposition}
Suppose $\varphi $ defines a $G_{2}$-structure. Then $\Delta \varphi $ has
the following type decomposition.%
\begin{subequations}%
\label{lapdecom}%
\begin{eqnarray}
\pi _{1}\Delta \varphi  &=&\frac{2}{7}\left( -2\func{Tr}\left( \func{curl}%
T\right) +2\func{Tr}\left( T^{t}T\right) +\left( \func{Tr}T\right) ^{2}-%
\func{Tr}\left( T^{2}\right) -T_{ab}T_{cd}\psi ^{abcd}\right) \varphi 
\label{pi1lapphi} \\
\pi _{7}\Delta \varphi  &=&\left( -\func{div}T+T_{{}}^{ab}T_{\ a}^{c}\varphi
_{\ bc}^{e}+\left( \func{Tr}T\right) T_{ab}\varphi ^{abe}\right) \lrcorner
\psi   \label{pi7lapphi} \\
\pi _{27}\Delta \varphi  &=&\mathrm{i}_{\varphi }\left( -\frac{3}{7}\func{Tr}%
\left( \func{curl}T\right) g_{de}-3\left( \func{curl}T^{t}\right) _{\left(
de\right) }-3\psi _{\ \ (d}^{abc}T_{|ab}^{{}}T_{c|e)}^{{}}\right. 
\label{pi27lapphi} \\
&&+\frac{3}{14}\left( \left( \func{Tr}T\right) ^{2}-\func{Tr}\left(
T^{2}\right) -T_{ab}T_{cd}\psi ^{abcd}+2\func{Tr}\left( T^{t}T\right) +7\psi
_{mnpq}T^{mn}T^{pq}\right) g_{de}  \notag \\
&&\left. -3\left( T\circ T\right) _{\left( de\right) }-3\psi _{\ \
(d}^{abc}T_{\left\vert ab\right\vert }^{{}}T_{e)c}^{{}}-3T_{\
d}^{a}T_{ae}^{{}}\right)   \notag
\end{eqnarray}%
\end{subequations}%
\end{proposition}

\begin{proof}
This is a straightforward, but rather long calculation. We first expand $%
\Delta \varphi $ in terms of the covariant derivative%
\begin{eqnarray*}
\left( \Delta \varphi \right) _{abc} &=&\left( dd^{\ast }\varphi +d^{\ast
}d\varphi \right) _{abc} \\
&=&-3\nabla _{\lbrack a}\nabla ^{d}\varphi _{\left\vert d\right\vert
bc]}-4\nabla ^{d}\nabla _{\lbrack d}\varphi _{abc]}
\end{eqnarray*}%
and then apply the formula for the covariant derivative of $\varphi $ in
terms of $T_{ab}$ (\ref{codiffphi}). This gives 
\begin{equation*}
\left( \Delta \varphi \right) _{abc}=-3\nabla _{\lbrack a}\left( T^{de}\psi
_{\left\vert ed\right\vert bc]}\right) -4\nabla ^{d}\left( T_{[d}^{\ e}\psi
_{\left\vert e\right\vert abc]}\right) .
\end{equation*}%
Now, expanding further, and applying (\ref{psitorsion}), we get a full
expression for $\Delta \varphi $ in terms of $T$ and its derivatives. This
can then be projected onto the components of $\Lambda ^{3}$ to obtain the
decomposition (\ref{lapdecom}).
\end{proof}

\begin{definition}
Given a differential operator $P$, denote by $D_{\varphi }P\left( \chi
\right) $ its linearization at $\varphi $, evaluated at $\chi $:%
\begin{equation*}
D_{\varphi }P\left( \chi \right) =\lim_{t\longrightarrow 0}\left( \frac{%
P\left( \varphi +t\chi \right) -P\left( \varphi \right) }{t}\right)
\end{equation*}
\end{definition}

\begin{proposition}
\label{propgenop}Consider a non-linear differential operator $P_{\varphi }$,
which is defined by $\varphi $, acting on the $G_{2}$-structure $\varphi $,
given to leading order by 
\begin{eqnarray}
\pi _{7}\left( P_{\varphi }\varphi \right) _{a} &=&-\left( \func{div}%
T\right) _{a}+l.o.t  \label{pi7genop} \\
\pi _{1\oplus 27}\left( P_{\varphi }\varphi \right) _{ab} &=&a_{27}\left( 
\func{curl}T^{t}\right) _{\left( ab\right) }+b_{27}\left( \func{curl}%
T\right) _{\left( ab\right) }+a_{1}\func{Tr}\left( \func{curl}T\right)
g_{ab}+l.o.t.  \label{pi127genop}
\end{eqnarray}%
for some constants $a_{27}$, $b_{27}$ and $a_{1}$. Then the linearization of 
$P_{\varphi }$ at $\varphi $ is given by 
\begin{eqnarray}
\pi _{7}\left( D_{\varphi }P_{\varphi }\right) \left( \chi \right) _{b} &=&%
\func{curl}\left( \func{div}h\right) _{b}-\nabla ^{2}X_{b}+l.o.t
\label{pi7genoplin} \\
\pi _{1\oplus 27}\left( D_{\varphi }P_{\varphi }\right) \left( \chi \right)
_{de} &=&a_{27}\left( -\nabla _{a}\nabla _{m}h_{nb}\varphi _{\ \
(d}^{mn}\varphi _{\ \ e)}^{ab}\right)  \label{pi127genoplin} \\
&&+b_{27}\left( \nabla ^{2}h_{de}-\nabla _{(d}^{{}}\left( \func{div}h\right)
_{\ e)}+\left( \nabla _{a}^{{}}\nabla _{(d}^{{}}X_{\left\vert b\right\vert
}^{{}}\right) \varphi _{\ \ e)}^{ab}\right)  \notag \\
&&+a_{1}\left( \nabla ^{2}\left( \func{Tr}h\right) -\func{div}\left( \func{%
div}h\right) \right) g_{de}+l.o.t.  \notag
\end{eqnarray}
\end{proposition}

\begin{proof}
Suppose 
\begin{equation*}
\dot{\varphi}=\chi =X^{a}\psi _{amnp}+3h_{[m}^{a}\varphi _{np]a}^{{}}
\end{equation*}%
then from \cite{karigiannis-2007} we know that 
\begin{eqnarray*}
\dot{g}_{ab} &=&2h_{ab} \\
\dot{g}^{ab} &=&-2h^{ab}
\end{eqnarray*}%
and similarly as in (\ref{dttorsphi}), 
\begin{equation*}
\dot{T}_{ab}=T_{a}^{\ c}h_{cb}^{{}}-T_{a}^{\ c}X_{{}}^{d}\varphi
_{dcb}^{{}}-\left( \func{curl}h\right) _{ab}+\nabla _{a}X_{b}.
\end{equation*}%
Then consider the linearization of $P$ at $\varphi $%
\begin{equation}
\pi _{7}\left( DP_{\varphi }\right) \left( \chi \right) =\func{div}\left( 
\func{curl}h\right) _{b}-\nabla ^{2}X_{b}+l.o.t.  \label{pi7genopa}
\end{equation}%
Note however, that 
\begin{eqnarray*}
\func{div}\left( \func{curl}h\right) _{b} &=&\nabla ^{a}\left( \left( \nabla
_{m}^{{}}h_{an}^{{}}\right) \varphi _{b}^{\ mn}\right) \\
&=&\left( \nabla ^{a}\nabla _{m}^{{}}h_{an}^{{}}\right) \varphi _{b}^{\
mn}+l.o.t. \\
&=&\nabla _{m}\left( \left( \func{div}h\right) _{n}\right) \varphi _{b}^{\
mn}-\left( R_{\ \ am}^{ca\ }h_{cn}^{{}}+R_{\ nam}^{c}h_{\ n}^{a}\right)
\varphi _{b}^{\ nm}+l.o.t. \\
&=&\func{curl}\left( \func{div}h\right) _{b}+l.o.t.
\end{eqnarray*}%
where we have used the Ricci identity for the Riemann tensor in the second
to last line. Hence, (\ref{pi7genopa}) gives us (\ref{pi7genoplin}). Now let
us look at the $\pi _{1\oplus 27}$ component. Substituting $\dot{T}$ into (%
\ref{pi127genop}) we get 
\begin{eqnarray*}
\pi _{1\oplus 27}\left( DP_{\varphi }\right) \left( \chi \right)
&=&a_{27}\left( -\nabla _{a}^{{}}\nabla _{m}^{{}}h_{nb}^{{}}\varphi _{\ \
(d}^{mn}\varphi _{\ \ e)}^{ab}+\nabla _{m}^{{}}\nabla
_{n}^{{}}X_{(d}^{{}}\varphi _{\ \ \ e)}^{mn}\right) \\
&&+a_{1}\left( \nabla _{a}^{{}}\nabla _{m}^{{}}h_{nb}^{{}}\varphi _{\ \ \
p}^{mn}\varphi _{{}}^{abp}+\nabla _{a}\nabla _{b}X_{c}\varphi _{\ \ \
}^{abc}\right) g_{de} \\
&&+b_{27}\left( -\nabla _{a}^{{}}\nabla _{m}^{{}}h_{n(d}^{{}}\varphi _{\ \
\left\vert b\right\vert }^{mn}\varphi _{\ \ e)}^{ab}+\nabla _{a}^{{}}\nabla
_{(d}^{{}}X_{\left\vert b\right\vert }^{{}}\varphi _{\ \ e)}^{ab}\right)
+l.o.t. \\
&=&a_{27}\left( -\nabla _{a}^{{}}\nabla _{m}^{{}}h_{nb}^{{}}\varphi _{\ \
(d}^{mn}\varphi _{\ \ e)}^{ab}\right) \\
&&+b_{27}\left( \nabla ^{2}h_{de}-\nabla _{a}^{{}}\nabla _{(d}^{{}}h_{\
e)}^{a}+\left( \nabla _{a}^{{}}\nabla _{(d}^{{}}X_{\left\vert b\right\vert
}^{{}}\right) \varphi _{\ \ e)}^{ab}\right) \\
&&+a_{1}\left( \nabla ^{2}\left( \func{Tr}h\right) -\nabla _{a}\nabla
_{b}h^{ab}\right) g_{de}+l.o.t.
\end{eqnarray*}%
Using the Ricci identity again to switch the order of the derivatives in the 
$b_{27}$ term, we get (\ref{pi127genoplin}).
\end{proof}

Note that the Laplacian of a general $G_{2}$-structure $\varphi $ is given
by the above operator $P$ with $a_{1}=-\frac{3}{7}$, $a_{27}=-3$ and $%
b_{27}=0$. As an example, consider the well-studied case of the Laplacian $%
\Delta _{\varphi }$ of the closed $3$-form $\varphi $. In this case, since $%
\varphi $ is closed, $T=\tau _{14}$ only has a component in the $14$%
-dimensional representation, and is hence anti-symmetric. Moreover, it is a
known fact (see e.g. \cite{GrigorianG2Torsion1}), that in this case 
\begin{equation}
d^{\ast }\tau _{14}=0  \label{T14cond}
\end{equation}%
and hence the highest order term $\func{div}T$ in (\ref{pi7lapphi})
vanishes. Moreover, since $\tau _{14}\in \Lambda _{14}^{2}$, we also have $%
\left( \tau _{14}\right) _{ab}\varphi ^{abc}=0$ (i.e. the projection to $%
\Lambda _{7}^{2}$ vanishes). Thus the leading order term $\func{Tr}\left( 
\func{curl}T\right) $ in (\ref{pi1lapphi}) becomes:%
\begin{eqnarray*}
\func{Tr}\left( \func{curl}T\right)  &=&-\left( \nabla _{a}T_{bc}\right)
\varphi ^{abc} \\
&=&-\nabla _{a}\left( \left( \tau _{14}\right) _{bc}\varphi ^{abc}\right)
+\left( \tau _{14}\right) _{bc}\nabla _{a}\varphi ^{abc} \\
&=&-\left( \tau _{14}\right) _{bc}\left( \tau _{14}\right) _{ea}\psi ^{eabc}
\\
&=&2\left( \tau _{14}\right) _{bc}\left( \tau _{14}\right) ^{bc}
\end{eqnarray*}%
where we have used the identity 
\begin{equation}
\omega _{ab}\psi _{\ \ \ cd}^{ab}=-2\omega _{cd}  \label{l14psiid}
\end{equation}%
for any $\omega \in \Lambda _{14}^{2}$. Overall, after similarly simplifying
other terms, we obtain 
\begin{subequations}%
\label{lapphiclosed} 
\begin{eqnarray}
\pi _{1}\Delta _{\varphi }\varphi  &=&\frac{2}{7}\left( \tau _{14}\right)
_{ab}\left( \tau _{14}\right) ^{ab}  \label{pi1lapphiclosed} \\
\pi _{7}\Delta _{\varphi }\varphi  &=&0  \label{pi7lapphiclosed} \\
\pi _{27}\Delta _{\varphi }\varphi  &=&3\func{curl}\left( \tau _{14}\right)
_{\left( ad\right) }-\frac{9}{7}g_{ad}\left( \tau _{14}\right) _{bc}\left(
\tau _{14}\right) ^{bc}+3\left( \tau _{14}\right) _{\ a}^{b}\left( \tau
_{14}\right) _{bd}  \label{pi27lapphiclosed}
\end{eqnarray}%
\end{subequations}%
Thus comparing the highest order terms we find that $\Delta \varphi $ in
this case corresponds to the operator $P$ in Proposition \ref{propgenop}
with $a_{1}=a_{27}=0$ and $b_{27}=3$. Hence, from Proposition \ref{propgenop}
we get the linearization: 
\begin{eqnarray}
\pi _{7}\left( D_{\varphi }\Delta _{\varphi }\right) \left( \chi \right) 
&=&0 \\
\pi _{1\oplus 27}\left( D_{\varphi }\Delta _{\varphi }\right) \left( \chi
\right)  &=&3\left( \nabla ^{2}h_{de}-\nabla _{(d}\left( \func{div}h\right)
_{e)}+\nabla _{(d}\left( \func{curl}X\right) _{e)}\right) +l.o.t.
\label{lapphiclosedlin}
\end{eqnarray}%
Let us now assume that $\chi $ is closed, and is of the form $\chi
=X\lrcorner \psi +3\mathrm{i}_{\varphi }\left( h\right) $. Then from
Proposition \ref{PropComponents}, we know the type decomposition of $d\chi $
and $d^{\ast }d\chi $ up to torsion terms. Since all of these components
have to be zero, we have the following relations%
\begin{subequations}%
\label{dchiclosedcond}%
\begin{eqnarray}
0 &=&\nabla _{a}\left( \func{div}X\right) +l.o.t.  \label{dchiclosedcond1} \\
0 &=&\nabla _{a}\func{Tr}h-\left( \func{div}h\right) _{a}-\left( \func{curl}%
X\right) _{a}+l.o.t.  \label{dchiclosedcond2} \\
0 &=&\nabla ^{2}\func{Tr}h-\func{div}\left( \func{div}h\right) +l.o.t. \\
0 &=&\nabla ^{2}X_{a}-\left( \func{curl}\left( \func{div}h\right) \right)
_{a}+l.o.t.
\end{eqnarray}%
\end{subequations}%
Also note for a vector field $v$, we have 
\begin{equation*}
d\left( v\lrcorner \varphi \right) _{abc}=3\nabla _{\lbrack a}\left(
v^{d}\varphi _{bc]d}\right) =3\left( \nabla _{\lbrack a}v^{d}\right) \varphi
_{bc]d}+l.o.t.
\end{equation*}%
Then, the symmetric part of $\nabla v$ simply gives the $\Lambda
_{1}^{3}\oplus \Lambda _{27}^{3}$ component of $d\left( v\lrcorner \varphi
\right) $, and the $\Lambda _{7}^{2}$ part of $\nabla v$ gives rise to the $%
\Lambda _{7}^{3}$ component of $d\left( v\lrcorner \varphi \right) $. In
fact, using identities in \cite{karigiannis-2007}, this can be re-written as 
\begin{equation}
d\left( v\lrcorner \varphi \right) =\frac{1}{2}\left( \func{curl}v\right)
\lrcorner \psi +3\mathrm{i}_{\varphi }\left( \nabla _{(m}v_{n)}\right)
+l.o.t.  \label{dlam2-7}
\end{equation}%
Now let $Y=\nabla \func{Tr}h$ and $Z=\func{curl}X$. Then, from (\ref{dlam2-7}%
), we get%
\begin{eqnarray}
\pi _{7}d\left( Y\lrcorner \varphi \right) _{a} &=&\frac{1}{2}\nabla
_{m}\left( \nabla _{n}\func{Tr}h\right) \varphi _{\ \ \ a}^{mn}=l.o.t. \\
\pi _{1\oplus 27}d\left( Y\lrcorner \varphi \right) _{ab} &=&3\left( \nabla
_{(a}\nabla _{b)}\func{Tr}h\right)   \notag \\
&=&3\left( \left( \nabla _{(a}^{{}}\nabla _{m}^{{}}X_{n}^{{}}\right) \varphi
_{\ \ \ b)}^{mn}+\nabla _{(a}^{{}}\nabla _{\left\vert r\right\vert
}^{{}}h_{\ b)}^{r}\right) +l.o.t.
\end{eqnarray}%
where we have applied (\ref{dchiclosedcond2}) for the second relation.
Similarly, we find that 
\begin{eqnarray}
\pi _{7}d\left( Z\lrcorner \varphi \right) _{p} &=&\frac{1}{2}\left( \nabla
_{d}\nabla _{a}X_{b}\right) \varphi _{\ \ e}^{ab}\varphi _{\ \ p}^{de}+\text{%
$l.o.t.$}  \notag \\
&=&\frac{1}{2}\left( \nabla _{d}\nabla _{a}X_{b}\right) \left(
-g_{{}}^{ad}g_{\ p}^{d}+g_{\ p}^{a}g^{bd}-\psi _{\ \ \ p}^{abd}\right) +%
\text{$l.o.t.$}  \notag \\
&=&-\frac{1}{2}\nabla ^{2}X_{p}+\frac{1}{2}\nabla _{p}\left( \nabla
_{b}X^{b}\right) +\frac{1}{2}Ric_{pb}X^{b}+\text{$l.o.t.$}  \notag \\
&=&-\frac{1}{2}\nabla ^{2}X_{p}+\text{$l.o.t.$} \\
\pi _{1\oplus 27}d\left( Z\lrcorner \varphi \right)  &=&3\mathrm{i}_{\varphi
}\left( \left( \nabla _{(d}\nabla _{|a}X_{b|}\right) \varphi _{\ \
e)}^{ab}\right) +\text{$l.o.t.$}
\end{eqnarray}%
where for the first relation we have used (\ref{phiphi1}) to contract the $%
\varphi $ terms, and then used (\ref{dchiclosedcond1}) and the fact that the
Ricci tensor depends only on the torsion.

Now note that 
\begin{equation}
2d\left( Z\lrcorner \varphi \right) -d\left( Y\lrcorner \varphi \right)
=-\left( \nabla ^{2}X\right) \lrcorner \psi +3\mathrm{i}_{\varphi }\left(
-\nabla _{(d}^{{}}\nabla _{\left\vert r\right\vert }^{{}}h_{\ e)}^{r}+\left(
\nabla _{(d}^{{}}\nabla _{|a}^{{}}X_{b|}^{{}}\right) \varphi _{\ \
e)}^{ab}\right) +\text{$l.o.t.$}
\end{equation}%
So in fact, 
\begin{equation}
D_{\varphi }\Delta _{\varphi }\left( \chi \right) =\left( \nabla
^{2}X\right) \lrcorner \psi +3\mathrm{i}_{\varphi }\left( \nabla
^{2}h_{de}\right) +2d\left( Z\lrcorner \varphi \right) -d\left( Y\lrcorner
\varphi \right) +l.o.t.  \label{dlapclosedlin}
\end{equation}%
It is easy to see that for any $3$-form $\chi $, 
\begin{equation}
\Delta _{\varphi }\chi =-\left( \nabla ^{2}X\right) \lrcorner \psi -3\mathrm{%
i}_{\varphi }\left( \nabla ^{2}h_{de}\right) +\text{$l.o.t.$}  \label{lapchi}
\end{equation}%
Since $\chi $ is closed, $\Delta _{\varphi }\chi $ is exact, so the lower
order terms in (\ref{dlapclosedlin}) are exact, so we rewrite them as an
exterior derivative of some $2$-form valued algebraic function of $\chi $.
Thus from (\ref{dlapclosedlin}) we conclude that 
\begin{equation}
D_{\varphi }\Delta _{\varphi }\left( \chi \right) =-\Delta _{\varphi }\chi
+2d\left( Z\lrcorner \varphi \right) -d\left( Y\lrcorner \varphi \right)
+dF\left( \chi \right)
\end{equation}%
Hence we have derived the same result as in \cite{BryantXu, XuYe}:

\begin{proposition}[\protect\cite{BryantXu, XuYe}]
If $\varphi $ is a closed $G_{2}$-structure, then the linearization of the
Laplacian $\Delta _{\varphi }$ at $\varphi $, evaluated at a closed form $%
\chi $ is given by 
\begin{equation}
D_{\varphi }\Delta _{\varphi }\left( \chi \right) =-\Delta _{\varphi }\chi -%
\mathcal{L}_{V_{\chi }}\varphi +dF\left( \chi \right)  \label{dlaplin}
\end{equation}%
where $V_{\varphi }=-2\func{curl}X+\nabla \func{Tr}h$ and $F\left( \chi
\right) $ is a $2$-form valued algebraic function of $\chi $.
\end{proposition}

The representation of $D_{\varphi }\Delta _{\varphi }\left( \chi \right) $
as (\ref{dlaplin}) shows us that we can apply the DeTurck trick - adding $%
\mathcal{L}_{V_{\varphi }}\varphi $ to (\ref{dlaplin}) gives a flow that is
elliptic in the direction of closed forms. In \cite{BryantXu, XuYe} this is
then used to show short-time existence and uniqueness for the Laplacian flow
for $3$-forms (\ref{orig-flow}).

Now let us consider the Laplacian of $\varphi $ for a co-closed $G_{2}$%
-structure. In this case, $T_{ab}$ is now symmetric. We can now write 
\begin{equation}
d\varphi =3\ast \mathrm{i}_{\varphi }\left( -T+\frac{1}{3}\left( \func{Tr}%
T\right) g\right) .  \label{dphicc}
\end{equation}%
From the condition $d^{2}\varphi =0$, or equivalently, $\left( d^{\ast
}\right) ^{2}\psi =0$, the torsion $T$ satisfies certain Bianchi-type
identities.

\begin{lemma}
\label{LemSymTConds}Suppose $\varphi $ is a co-closed $G_{2}$-structure.
Then the torsion tensor $T$ satisfies the following identities 
\begin{eqnarray}
\func{div}T &=&\nabla \func{Tr}T  \label{Tsymcond} \\
\left( \func{curl}T\right) _{[ab]} &=&0  \label{Tsymcond2}
\end{eqnarray}
\end{lemma}

\begin{proof}
We have the condition that $d^{\ast }\left( 3\mathrm{i}_{\varphi }\left( -T+%
\frac{1}{3}\left( \func{Tr}T\right) g\right) \right) =0$. Set $h=-T+\frac{1}{%
3}\left( \func{Tr}T\right) g$ and $X=0$ in Proposition \ref{PropComponents}.
Then from (\ref{d3scomp7}), we get (\ref{Tsymcond}) and (\ref{Tsymcond2})
then follows from (\ref{d3scompall}).
\end{proof}

\begin{proposition}
Suppose $\varphi $ is a co-closed $G_{2}$-structure, then the type
decomposition of $\Delta _{\varphi }\varphi $ is given by%
\begin{subequations}%
\label{lapdecomsym} 
\begin{eqnarray}
\pi _{1}\Delta _{\varphi }\varphi &=&\frac{2}{7}\left\vert T\right\vert ^{2}+%
\frac{2}{7}\left( \func{Tr}T\right) ^{2} \\
\left( \pi _{7}\Delta _{\varphi }\varphi \right) &=&-\func{div}T=-\nabla 
\func{Tr}\left( T\right) \\
\left( \pi _{27}\Delta _{\varphi }\varphi \right) _{ad} &=&-3\left( \func{%
curl}T\right) _{ad}-\frac{3}{2}\left( T\circ T\right) _{ad}-3T_{ab}^{{}}T_{\
d}^{b} \\
&&+\frac{3}{14}g_{ad}\left\vert T\right\vert ^{2}+\frac{3}{14}g_{ad}\left( 
\func{Tr}T\right) ^{2}  \notag
\end{eqnarray}%
\end{subequations}%
\end{proposition}

\begin{proof}
We obtain this directly from (\ref{lapdecom}) by using the fact that $T$ is
symmetric and also applying identity (\ref{Tsymcond}). Note that from (\ref%
{Tsymcond2}) it follows that the symmetric part of $\func{curl}T$ is
actually equal to $\func{curl}T$.
\end{proof}

Now suppose we have the flow 
\begin{equation}
\frac{d\psi }{dt}=\Delta _{\psi }\psi =\ast _{\varphi }\Delta _{\varphi
}\varphi .  \label{lapcoflow1}
\end{equation}%
Then from (\ref{lapdecomsym}), 
\begin{eqnarray}
\Delta _{\psi }\psi &=&d\func{Tr}\left( T\right) \wedge \varphi +\left( 
\frac{2}{7}\left\vert T\right\vert ^{2}+\frac{2}{7}\left( \func{Tr}T\right)
^{2}\right) \psi  \notag \\
&&+3\ast \mathrm{i}_{\varphi }\left( -\left( \func{curl}T\right) _{ad}-\frac{%
1}{2}\left( T\circ T\right) _{ad}-T_{ab}^{{}}T_{\ d}^{b}\right.  \notag \\
&&\left. +\frac{1}{14}g_{ad}\left\vert T\right\vert ^{2}+\frac{1}{14}%
g_{ad}\left( \func{Tr}T\right) ^{2}\right)  \notag \\
&=&d\func{Tr}\left( T\right) \wedge \varphi  \label{lappsi} \\
&&+3\ast \mathrm{i}_{\varphi }\left( -\left( \func{curl}T\right) _{ad}-\frac{%
1}{2}\left( T\circ T\right) _{ad}-T_{ab}^{{}}T_{\ d}^{b}\right.  \notag \\
&&\left. +\frac{1}{6}g_{ad}\left\vert T\right\vert ^{2}+\frac{1}{6}%
g_{ad}\left( \func{Tr}T\right) ^{2}\right)  \notag
\end{eqnarray}%
The corresponding evolution of the metric now follows from Proposition \ref%
{propflow} - in this case, 
\begin{eqnarray*}
h_{ad} &=&-\left( \func{curl}T\right) _{ad}-\frac{1}{2}\left( T\circ
T\right) _{ad}-T_{ab}^{{}}T_{\ d}^{b}+\frac{1}{6}g_{ad}\left\vert
T\right\vert ^{2}+\frac{1}{6}g_{ad}\left( \func{Tr}T\right) ^{2} \\
\func{Tr}h &=&\frac{2}{3}\left( \left\vert T\right\vert ^{2}+\left( \func{Tr}%
T\right) ^{2}\right)
\end{eqnarray*}%
hence the evolution of the metric becomes%
\begin{eqnarray}
\frac{d}{dt}g_{ab} &=&\frac{1}{2}\left( \func{Tr}h\right) g_{ab}-2h_{ab} 
\notag \\
&=&2\left( \func{curl}T\right) _{ab}+\left( T\circ T\right)
_{ab}+2T_{ac}^{{}}T_{\ b}^{c}  \label{coclosedgdt}
\end{eqnarray}%
Note that for a co-closed $G_{2}$, the Ricci curvature is given by 
\begin{equation}
R_{ad}=-\left( \func{curl}T\right) _{ad}-T_{ab}^{{}}T_{\ d}^{b}+\left( \func{%
Tr}T\right) T_{ad}.  \label{coclosedRicci}
\end{equation}%
The details can be found in \cite{GrigorianG2Torsion1,karigiannis-2007}.
Hence we can rewrite (\ref{coclosedgdt}) as 
\begin{equation}
\frac{d}{dt}g_{ab}=-2R_{ab}+\left( T\circ T\right) _{ab}+2\left( \func{Tr}%
T\right) T_{ad}.  \label{coclosedgdt2}
\end{equation}%
Therefore, same as for the Laplacian flow of the $3$-form $\varphi $ \cite%
{bryant-2003}, the leading term of the metric flow corresponds to the Ricci
flow. Also, the evolution of the volume form under this flow is given by 
\begin{equation}
\frac{d}{dt}\sqrt{\det g}=\frac{1}{2}\left( \left\vert T\right\vert
^{2}+\left( \func{Tr}T\right) ^{2}\right) \sqrt{\det g}
\label{cocloseddetgdt}
\end{equation}%
For non-trivial torsion this is always positive and is zero if and only if $%
T=0$. Hence the volume functional $V$ (\ref{volfuncpsi}) grows monotonically
along the flow (\ref{lapcoflow1}), with an extremum being reached along a
flow line if and only if the torsion vanishes. We have already seen in
Proposition \ref{propvolextr} that it is generally true that torsion-free $%
G_{2}$-structures correspond to critical points of $V$ when restricted to a
fixed cohomology class of $\varphi $. The Laplacian flow of $\varphi $ (\ref%
{orig-flow}) shares this property. As shown in \cite{bryant-2003}, the
volume form grows monotonically along the flow (\ref{orig-flow}) and in \cite%
{BryantXu}, it was interpreted as the gradient flow of the volume functional
with respect to an unusual metric. Similarly, we can do the same for (\ref%
{lapcoflow1}). As in \cite{BryantXu}, define the following metric on $%
dC^{\infty }\left( M,\Lambda ^{3}\left( M\right) \right) $:%
\begin{equation*}
\left\langle \chi _{1},\chi _{2}\right\rangle _{\psi }=\frac{1}{4}%
\int_{M}G_{\psi }\chi _{1}\wedge \ast \chi _{2}.
\end{equation*}%
for any exact $4$-forms $\chi _{1}$ and $\chi _{2}$. As before, $G_{\psi }$
is the Green's operator for the Hodge Laplacian $\Delta _{\psi }$. From (\ref%
{delV}), we know that for an exact $4$-form $\dot{\psi}$, the deformation of 
$V$ is given by 
\begin{eqnarray*}
\delta V\left( \dot{\psi}\right) &=&\frac{1}{4}\int_{M}\dot{\psi}\wedge
\varphi =\frac{1}{4}\int_{M}\Delta _{\psi }G_{\psi }\dot{\psi}\wedge \varphi
\\
&=&\frac{1}{4}\int_{M}G_{\psi }\dot{\psi}\wedge \ast \Delta _{\psi }\psi \\
&=&\left\langle \dot{\psi},\Delta _{\psi }\psi \right\rangle _{\psi }
\end{eqnarray*}%
Hence indeed, the gradient flow of $V$ is given by (\ref{lapcoflow1}).

Now consider the linearization of $\Delta _{\psi }\psi $.

\begin{proposition}
The linearization of $\Delta _{\psi }$ at $\psi $ is given by 
\begin{eqnarray}
\pi _{7}\left( D_{\psi }\Delta _{\psi }\right) \left( \chi \right)
&=&d\left( \func{div}X\right) \wedge \varphi +l.o.t. \\
\pi _{1\oplus 27}\left( D_{\psi }\Delta _{\psi }\right) \left( \chi \right)
&=&3\ast \mathrm{i}_{\varphi }\left( \nabla ^{2}h_{de}-\nabla
_{(d}^{{}}\nabla _{\left\vert a\right\vert }^{{}}h_{\ e)}^{a}-\left( \nabla
_{a}^{{}}\nabla _{(d}^{{}}X_{\left\vert b\right\vert }^{{}}\right) \varphi
_{\ \ e)}^{ab}\right.  \label{pi127lappsilin} \\
&&\left. +\frac{1}{4}\nabla _{a}\nabla _{d}\func{Tr}h-\frac{1}{4}\left(
\nabla ^{2}\func{Tr}h\right) g_{ad}\right) +l.o.t.  \notag
\end{eqnarray}%
where $\chi $ is a $4$-form given by 
\begin{equation*}
\chi =\ast \left( X\lrcorner \psi +3\mathrm{i}_{\varphi }\left( h\right)
\right) .
\end{equation*}%
Moreover, if $\chi $ is a closed form, then the $1\oplus 27$ component of $%
D_{\psi }\Delta _{\psi }\left( \chi \right) $ can written as 
\begin{eqnarray}
\pi _{1\oplus 27}\left( D_{\psi }\Delta _{\psi }\right) \left( \chi \right)
&=&\frac{3}{2}\ast \mathrm{i}_{\varphi }\left( \nabla ^{2}h_{de}-\nabla
_{(d}^{{}}\nabla _{\left\vert a\right\vert }^{{}}h_{\ e)}^{a}-\nabla
_{a}^{{}}\nabla _{(d}^{{}}X_{\left\vert b\right\vert }^{{}}\varphi _{\ \
e)}^{ab}\right.  \label{pi127lappsilin3} \\
&&\left. -\nabla _{b}^{{}}\nabla _{m}^{{}}h_{cn}\varphi _{\ \
(d}^{bc}\varphi _{\ \ \ e)}^{mn}\right) +l.o.t.  \notag
\end{eqnarray}
\end{proposition}

\begin{proof}
To find the linearization, we just take the decomposition of $\Delta _{\psi
}\psi $ (\ref{lappsi}) and use the formula (\ref{dttorspsi}) for the
variation of $T_{ab}$. So we have 
\begin{equation}
\dot{T}_{ab}=\frac{1}{4}\left( \func{Tr}h\right) T_{ab}-T_{a}^{\
c}h_{cb}^{{}}-T_{a}^{\ c}X_{{}}^{d}\varphi _{dcb}^{{}}+\left( \func{curl}%
h\right) _{ab}+\nabla _{a}X_{b}-\frac{1}{4}\left( \nabla _{c}\func{Tr}%
h\right) \varphi _{\ ab}^{c}
\end{equation}%
Now,%
\begin{eqnarray*}
\pi _{7}\left( D_{\psi }\Delta _{\psi }\right) \left( \chi \right)
&=&d\left( \func{div}X\right) \wedge \varphi +l.o.t. \\
\pi _{1\oplus 27}\left( D_{\psi }\Delta _{\psi }\right) \left( \chi \right)
&=&3\ast \mathrm{i}_{\varphi }\left( -\nabla _{a}^{{}}\nabla
_{m}^{{}}h_{n(d}^{{}}\varphi _{\ \ \ \left\vert b\right\vert }^{mn}\varphi
_{\ \ e)}^{ab}-\nabla _{a}^{{}}\nabla _{(d}^{{}}X_{\left\vert b\right\vert
}^{{}}\varphi _{\ \ e)}^{ab}\right. \\
&&+\left. \frac{1}{4}\nabla _{c}^{{}}\left( \nabla _{m}^{{}}\func{Tr}%
h\right) \varphi _{\ \ (a\left\vert b\right\vert }^{m}\varphi _{d)}^{\
cb}\right) +l.o.t. \\
&=&3\ast \mathrm{i}_{\varphi }\left( \nabla ^{2}h_{de}-\nabla
_{(d}^{{}}\nabla _{\left\vert a\right\vert }^{{}}h_{\ e)}^{a}-\nabla
_{a}^{{}}\nabla _{(d}^{{}}X_{\left\vert b\right\vert }^{{}}\varphi _{\ \
e)}^{ab}\right. \\
&&+\left. \frac{1}{4}\nabla _{a}\nabla _{d}\func{Tr}h-\frac{1}{4}\left(
\nabla ^{2}\func{Tr}h\right) g_{ad}\right) +l.o.t.
\end{eqnarray*}%
where we have twice used the contraction identity (\ref{phiphi1}) in the
last line.

To get the expression (\ref{pi127lappsilin3}), note that when $\chi $ is
closed, $\dot{T}_{ab}$ is symmetric, since the torsion class remains in $%
W_{1}\oplus W_{27}$. Hence the antisymmetric part of (\ref{dttorspsi})
vanishes and we can write%
\begin{equation}
\dot{T}_{ab}=\frac{1}{4}\left( \func{Tr}h\right) T_{ab}-T_{(a}^{\ \
c}h_{cb)}^{{}}-T_{(a}^{\ \ c}X^{d}\varphi _{b)dc}+\left( \func{curl}h\right)
_{\left( ab\right) }+\nabla _{(a}X_{b)}  \label{dttorssym}
\end{equation}%
It can been seen explicitly from Proposition \ref{PropComponents} that the
antisymmetric part of $\dot{T}$ is equal precisely to $\ast d\chi $, which
is of course zero. Now plugging (\ref{dttorssym}) into (\ref{lappsi}) we
find that $\pi _{7}\left( D_{\psi }\Delta _{\psi }\right) \left( \chi
\right) $ remains the same, and $\pi _{1\oplus 27}\left( D_{\psi }\Delta
_{\psi }\right) \left( \chi \right) $ becomes as in (\ref{pi127lappsilin3}).
\end{proof}

To see that this is not a positive operator, consider the principal symbol $%
\sigma _{\xi }\left( D_{\psi }\Delta _{\psi }\right) \left( \chi \right) $
for some vector $\xi :$%
\begin{eqnarray*}
\pi _{7}\sigma _{\xi }\left( D_{\psi }\Delta _{\psi }\right) \left( \chi
\right) &=&\left\langle X,\xi \right\rangle \xi ^{\flat }\wedge \varphi \\
\pi _{1\oplus 27}\sigma _{\xi }\left( D_{\psi }\Delta _{\psi }\right) \left(
\chi \right) &=&3\ast \mathrm{i}_{\varphi }\left( \left\vert \xi \right\vert
^{2}h_{de}-\xi _{(d}^{{}}\xi _{\left\vert a\right\vert }^{{}}h_{\
e)}^{a}-\xi _{a}^{{}}\xi _{(d}^{{}}X_{\left\vert b\right\vert }^{{}}\varphi
_{\ \ e)}^{ab}+\frac{1}{4}\xi _{a}\xi _{d}\func{Tr}h\right. \\
&&\left. -\frac{1}{4}\left\vert \xi \right\vert ^{2}\left( \func{Tr}h\right)
g_{ad}\right) +l.o.t.
\end{eqnarray*}%
Then, if $h=0$, we have 
\begin{equation*}
\left\langle \sigma _{\xi }\left( D_{\psi }\Delta _{\psi }\right) \left(
\chi \right) ,\chi \right\rangle =-4\left\langle X,\xi \right\rangle
^{2}\leq 0.
\end{equation*}

We may attempt to use DeTurck's trick to modify this operator by adding a
certain Lie derivative $\mathcal{L}_{V\left( \chi \right) }\psi $ to the
flow (\ref{lapcoflow1}), where $V\left( \chi \right) $ is linear in the
first derivatives of $\chi $, so that the modified flow of $\psi $ is now
given by 
\begin{equation}
\frac{d\psi }{dt}=\Delta _{\psi }\psi +\mathcal{L}_{V\left( \chi \right)
}\psi  \label{coflowmod1}
\end{equation}%
For convenience, denote 
\begin{equation}
Q_{\psi }\psi =\Delta _{\psi }\psi +\mathcal{L}_{V\left( \chi \right) }\psi .
\label{Qpsi}
\end{equation}%
Since $\psi $ is closed, we have 
\begin{equation*}
\mathcal{L}_{V\left( \chi \right) }\psi =d\left( V\left( \chi \right)
\lrcorner \psi \right) .
\end{equation*}%
Consider first the type decomposition of $d\left( V\left( \chi \right)
\lrcorner \psi \right) $. From Proposition \ref{PropComponents}, we have 
\begin{subequations}%
\label{dvpsidecomp}

\begin{eqnarray}
\pi _{7}d\left( V\lrcorner \psi \right) &=&-\frac{1}{2}\left( \func{curl}%
V\right) \wedge \varphi +\text{$l.o.t.$} \\
\pi _{1\oplus 27}d\left( V\lrcorner \psi \right) &=&3\ast \mathrm{i}%
_{\varphi }\left( -\nabla _{(m}V_{n)}+\frac{1}{3}\left( \func{div}V\right)
g_{mn}\right) +\text{$l.o.t.$}
\end{eqnarray}%
\end{subequations}%
We are only interested in linearizing (\ref{Qpsi}) in the direction of
closed $4$-forms, so suppose $\chi =\ast \left( X\lrcorner \psi +3\mathrm{i}%
_{\varphi }\left( h\right) \right) $ is closed, and hence $\ast \chi
=X\lrcorner \psi +3\mathrm{i}_{\varphi }\left( h\right) $ is co-closed. This
requirement gives a number of conditions on $X$ and $h$. From Proposition %
\ref{PropComponents} we then get the following 
\begin{subequations}%
\label{closed4formrels1} 
\begin{eqnarray}
-\left( \left( \func{div}h\right) \lrcorner \varphi \right) _{bc}-2\left( 
\func{curl}h\right) _{\left[ bc\right] }+\nabla _{m}^{{}}X_{n}^{{}}\psi _{\
\ \ bc}^{mn} &=&\text{$l.o.t.$}  \label{pi14cond} \\
\left( \func{curl}X\right) _{a}-\left( \func{div}h\right) _{a}-\frac{1}{2}%
\nabla _{a}^{{}}\func{Tr}h &=&\text{$l.o.t.$}  \label{pi7cond}
\end{eqnarray}%
\end{subequations}
The lower order terms in (\ref{closed4formrels1}) are linear in $X$ and $h$
and depend only on the torsion. By differentiating (\ref{closed4formrels}),
we get the following relations:%
\begin{subequations}%
\label{closed4formrels} 
\begin{eqnarray}
\func{div}\left( \func{div}h\right) +\frac{1}{2}\nabla ^{2}\func{Tr}h &=&%
\text{$l.o.t.$}  \label{pi1cond} \\
\nabla _{(a}\left( \func{div}h\right) _{b)}+\frac{1}{2}\nabla _{a}\nabla _{b}%
\func{Tr}h-\left( \nabla _{m}^{{}}\nabla _{(a}^{{}}X_{\left\vert
n\right\vert }^{{}}\right) \varphi _{\ \ \ b)}^{mn} &=&\text{$l.o.t.$}
\label{pi27cond} \\
\nabla _{a}\left( \func{div}X\right) -\nabla ^{2}X_{a}-\func{curl}\left( 
\func{div}h\right) _{a} &=&\text{$l.o.t.$}  \label{pi7cond2}
\end{eqnarray}%
\end{subequations}
Here, equation (\ref{pi1cond}) is obtained by taking the divergence of (\ref%
{pi7cond}). Equation (\ref{pi27cond}) is obtained by taking the curl of (\ref%
{pi14cond}), and then symmetrizing. Similarly, equation (\ref{pi7cond2}) is
obtained by taking the curl of (\ref{pi7cond}). In both cases the identities
(\ref{contids}) are used to simplify contractions of $\varphi $ and $\psi $%
.\ 

The only possible vector fields that are linear in covariant derivatives of $%
X$ and $h$ are 
\begin{eqnarray}
Y &=&\nabla \func{Tr}h \\
W &=&\func{div}h \\
Z &=&\func{curl}X
\end{eqnarray}%
However from (\ref{pi7cond}) we see that these vectors are linearly
dependent up to lower order terms, so it is sufficient to take a linear
combination of any two of them. For convenience we will choose $Y$ and $Z$.
First consider $d\left( Y\lrcorner \psi \right) $:%
\begin{subequations}
\begin{eqnarray}
\pi _{7}d\left( Y\lrcorner \psi \right) &=&\left( -\frac{1}{2}\nabla
_{m}^{{}}\nabla _{n}^{{}}\left( \func{Tr}h\right) \varphi _{\ \ \
a}^{mn}\right) \wedge \varphi +\text{$l.o.t.$}=\text{$l.o.t.$} \\
\pi _{1\oplus 27}d\left( Y\lrcorner \psi \right) &=&3\ast \mathrm{i}%
_{\varphi }\left( -\nabla _{m}\nabla _{n}\func{Tr}h+\frac{1}{3}\nabla
^{2}\left( \func{Tr}h\right) g_{mn}\right) +\text{$l.o.t.$}
\end{eqnarray}%
\end{subequations}%
Similarly, the decomposition of $d\left( Z\lrcorner \psi \right) $ is 
\begin{subequations}%
\begin{eqnarray}
\pi _{7}d\left( Z\lrcorner \psi \right) &=&\left( -\frac{1}{2}\nabla
_{m}^{{}}\nabla _{c}^{{}}X_{d}^{{}}\varphi _{\ \ n}^{cd}\varphi _{\ \ \
a}^{mn}\right) \wedge \varphi +\text{$l.o.t.$}  \notag \\
&=&\frac{1}{2}\left( \nabla ^{2}X_{a}-\nabla _{a}\left( \func{div}X\right)
\right) +\text{$l.o.t.$} \\
\pi _{1\oplus 27}d\left( Z\lrcorner \psi \right) &=&3\ast \mathrm{i}%
_{\varphi }\left( -\nabla _{m}^{{}}\nabla _{(a}^{{}}X_{n}^{{}}\varphi _{\ \
\ d)}^{mn}\right) +\text{$l.o.t.$}
\end{eqnarray}%
\end{subequations}%
Again, we have used the contraction identity (\ref{phiphi1}) to get to the
second line.

Hence, if we take 
\begin{equation}
V=\frac{3}{4}Y-2Z
\end{equation}%
and apply (\ref{pi27cond}), we find that the linearization of $Q_{\psi }$ is
now given 
\begin{subequations}%
\label{coflowlinmod1} 
\begin{eqnarray}
\pi _{7}\left( D_{\psi }Q_{\psi }\right) \left( \chi \right) &=&\left(
-\nabla ^{2}X+2d\left( \func{div}X\right) \right) \wedge \varphi +l.o.t.
\label{pi7linmodcoflow1} \\
\pi _{1\oplus 27}\left( D_{\psi }Q_{\psi }\right) \left( \chi \right)
&=&3\ast \mathrm{i}_{\varphi }\left( \nabla ^{2}h_{ad}\right) +l.o.t.
\label{pi127linmodcoflow1}
\end{eqnarray}%
\end{subequations}
For some vector $\xi $, consider the principal symbol $\sigma _{\xi }\left(
D_{\psi }Q_{\psi }\right) \chi ,$ then 
\begin{eqnarray}
\left\langle \sigma _{\xi }\left( D_{\psi }Q_{\psi }\right) \chi ,\chi
\right\rangle &=&\left\langle \left( -\left\vert \xi \right\vert ^{2}X+2\xi
\left\langle \xi ,X\right\rangle \right) \wedge \varphi ,-X\wedge \varphi
\right\rangle  \notag \\
&&+9\left\langle \mathrm{i}_{\varphi }\left( \left\vert \xi \right\vert
^{2}h_{ad}\right) ,\mathrm{i}_{\varphi }\left( h_{ad}\right) \right\rangle 
\notag \\
&=&4\left( \left\vert \xi \right\vert ^{2}\left\vert X\right\vert
^{2}-2\left\langle \xi ,X\right\rangle ^{2}\right) +2\left\vert \xi
\right\vert ^{2}\left\vert h\right\vert ^{2}  \label{princsymbQ}
\end{eqnarray}%
Hence we see that the principal symbol is still indefinite, and so (\ref%
{coflowmod1}) is not parabolic. Note that adding more instances of $\mathcal{%
L}_{Z}\psi $ would not help. To see this, let $\sigma _{\xi }\left( d\left(
Z\lrcorner \psi \right) \right) $ be the principal symbol of $d\left(
Z\lrcorner \psi \right) $ for some vector $\xi $. Using (\ref{pi7cond2}), we
get 
\begin{equation}
\left\langle \sigma _{\xi }\left( d\left( Z\lrcorner \psi \right) \right)
,\chi \right\rangle =-2\xi _{m}^{{}}\xi _{{}}^{n}h_{np}^{{}}\varphi _{\ \ \
a}^{mp}X_{{}}^{a}-2\xi _{m}^{{}}\xi _{a}^{{}}X_{n}^{{}}\varphi _{\ \ \
d}^{mn}h_{{}}^{ad}=0
\end{equation}%
Therefore, $\mathcal{L}_{Z}\psi $ would not contribute in any way to (\ref%
{princsymbQ}). However, overall, we have shown that we can rewrite the
linearized operator $\Delta _{\psi }\psi $ in the following way.

\begin{proposition}
\label{proplincoflow}The linearization of the operator $\Delta _{\psi }\psi $
evaluated at a closed $4$-form $\chi $ given by 
\begin{equation*}
\chi =\ast \left( X\lrcorner \psi +3\mathrm{i}_{\varphi }\left( h\right)
\right)
\end{equation*}%
is given by 
\begin{equation}
D_{\psi }\Delta _{\psi }\left( \chi \right) =-\Delta _{\psi }\chi -\mathcal{L%
}_{V\left( \chi \right) }\psi +2d\left( \left( \func{div}X\right) \varphi
\right) +dF\left( \chi \right)  \label{coflowlin}
\end{equation}%
where 
\begin{equation}
V\left( \chi \right) =\frac{3}{4}\nabla \func{Tr}h-2\func{curl}X
\label{coflowvect}
\end{equation}%
and $F\left( \chi \right) $ is a $3$-form-valued algebraic function of $\chi 
$.
\end{proposition}

\begin{proof}
From (\ref{coflowlinmod1}), and using (\ref{lapchi}) we have 
\begin{equation*}
D_{\psi }\Delta _{\psi }\left( \chi \right) =-\Delta _{\psi }\chi -\mathcal{L%
}_{V\left( \chi \right) }\psi +2d\left( \func{div}X\right) \wedge \varphi
+l.o.t.
\end{equation*}%
However we can write 
\begin{equation*}
\left( \nabla ^{2}X\right) \wedge \varphi =d\left( \left( \func{div}X\right)
\varphi \right) +l.o.t.
\end{equation*}%
and whence, 
\begin{equation}
D_{\psi }\Delta _{\psi }\left( \chi \right) =-\Delta _{\psi }\chi -\mathcal{L%
}_{V_{\chi }}\psi +2d\left( \left( \func{div}X\right) \varphi \right) +l.o.t.
\label{coflowlin2}
\end{equation}%
However, for a closed $\psi $, $\Delta _{\psi }\psi $ is exact, and now all
the higher order terms are exact, so the lower order terms must be an exact $%
4$-form, and can thus be represented as $dF\left( \chi \right) $ where $F$
is a $3$-form which depends algebraically on $\chi $.
\end{proof}

In both the closed and co-closed case the vector field $V_{\chi }$ has been
derived by considering linearizations of $\Delta _{\varphi }\varphi $ and $%
\Delta _{\psi }\psi $, respectively. It turns out that these vector fields,
as well as the vector field used in DeTurck's trick for Ricci flow, actually
have precisely the same origin, even though this is not clear from the above
situation, since the specifics of each case are a little different. In
general, following DeTurck's original idea \cite{DeTurckTrick}, suppose we
fix an arbitrary time-independent background metric $\bar{g}$. Given a
geometric flow, this could for example be taken to be the metric at time $%
t=0 $. With respect to this background metric, we can write down a
background Levi-Civita connection $\bar{\nabla}$, using which we get a
background Riemann curvature, and in particular, a background Ricci
curvature $\overline{Ric}$. Now given an evolving metric $g$, we can express
all the quantities related to $g$ in terms of the background quantities. Let 
\begin{equation*}
\bar{h}=g-\bar{g}.
\end{equation*}%
Then, it is easy to show that the difference between the corresponding
Christoffel symbols $\Gamma $ and $\bar{\Gamma}$ is given by 
\begin{equation}
\mathcal{T}_{a\ c}^{\ b}:=\Gamma _{a\ c}^{\ b}-\bar{\Gamma}_{a\ c}^{\ b}=%
\frac{1}{2}g^{bd}\left( \bar{\nabla}_{a}\bar{h}_{cd}+\bar{\nabla}_{c}\bar{h}%
_{ad}-\bar{\nabla}_{d}\bar{h}_{ac}\right)  \label{gammabar}
\end{equation}%
From this, it is possible to obtain an expression for $Ric$ (the details can
be found in \cite{CaoZhuRicci,MorganTianRicci}, for example) in terms of the
background quantities and difference $\bar{h}$%
\begin{equation}
Ric=\overline{Ric}-\frac{1}{2}\bar{\nabla}^{2}h+\frac{1}{2}\mathcal{L}_{\bar{%
V}}\bar{g}+O\left( \left\vert h\right\vert ^{2}\right)  \label{ricbar}
\end{equation}%
where the vector $\bar{V}$ is given by 
\begin{equation}
\bar{V}^{b}=g^{ac}\mathcal{T}_{a\ c}^{\ b}  \label{vbar}
\end{equation}%
In Ricci flow applications, the expression (\ref{ricbar}) is then used to
relate the Ricci flow to the strictly parabolic Ricci-DeTurck flow, and
hence show short-time existence.

Going back to our $4$-form flow (\ref{lapcoflow1}), suppose the background
metric is now $\bar{g}=g_{0}$ - the metric associated to the initial $G_{2}$%
-structure $\psi _{0}$. We also have 
\begin{equation*}
\psi -\psi _{0}=\chi
\end{equation*}%
where $\chi $ is some closed $4$-form given by $\chi =\ast \left( X\lrcorner
\psi +3\mathrm{i}_{\varphi }\left( h\right) \right) $, as before. From
Proposition \ref{propflow}, 
\begin{equation*}
g=\bar{g}+\left( \frac{1}{2}\left( \overline{\func{Tr}}h\right) \bar{g}%
-2h\right) +O\left( \left\vert \chi \right\vert ^{2}\right)
\end{equation*}%
where $\overline{\func{Tr}}h=\bar{g}^{ab}h_{ab}$. Thus from (\ref{gammabar}%
), 
\begin{eqnarray*}
\mathcal{T}_{a\ c}^{\ b} &=&\frac{1}{2}\bar{g}^{bd}\bar{\nabla}_{a}\left( 
\frac{1}{2}\left( \overline{\func{Tr}}h\right) \bar{g}_{cd}-2h_{cd}\right) +%
\frac{1}{2}\bar{g}^{bd}\bar{\nabla}_{c}\left( \frac{1}{2}\left( \overline{%
\func{Tr}}h\right) \bar{g}_{ad}-2h_{ad}\right) \\
&&-\frac{1}{2}\bar{g}^{bd}\bar{\nabla}_{d}\left( \frac{1}{2}\left( \overline{%
\func{Tr}}h\right) \bar{g}_{ac}-2h_{ac}\right) +O\left( \left\vert \chi
\right\vert ^{2}\right) \\
&=&-\bar{g}^{bd}\left( \bar{\nabla}_{a}h_{cd}+\bar{\nabla}_{c}h_{ad}-\bar{%
\nabla}_{d}h_{ac}\right) \\
&&+\frac{1}{4}\left( \left( \bar{\nabla}_{a}\overline{\func{Tr}}h\right)
\delta _{\ c}^{b}+\left( \bar{\nabla}_{c}\overline{\func{Tr}}h\right) \delta
_{\ a}^{b}-\left( \bar{\nabla}^{d}\overline{\func{Tr}}h\right) \bar{g}%
_{ac}\right) +O\left( \left\vert \chi \right\vert ^{2}\right)
\end{eqnarray*}%
Therefore, in this case the vector $\bar{V}$ becomes 
\begin{equation}
\bar{V}^{b}=-2\bar{\nabla}_{c}h_{d}^{\ c}-\frac{1}{4}\bar{\nabla}^{b}%
\overline{\func{Tr}}h+O\left( \left\vert \chi \right\vert ^{2}\right)
\label{vbar2}
\end{equation}%
where indices are raised using the background inverse metric $\bar{g}^{-1}$.
Note however that $\chi $ is closed, so it satisfies conditions (\ref%
{closed4formrels1}). Using (\ref{pi7cond}), we can re-write (\ref{vbar2}) as 
\begin{eqnarray}
\bar{V}^{b} &=&\frac{3}{4}\bar{\nabla}^{b}\overline{\func{Tr}}h-2\left( \bar{%
\nabla}_{a}X_{c}\right) \bar{\varphi}^{acb}+U^{b}\left( h,X\right) +O\left(
\left\vert \chi \right\vert ^{2}\right)  \notag \\
&=&V\left( \chi \right) ^{b}+U^{b}\left( h,X\right) +O\left( \left\vert \chi
\right\vert ^{2}\right)  \label{vbar3}
\end{eqnarray}%
where $U$ is a vector-valued function linear in $h$ and $X$ and which
depends only on the torsion of $\bar{\varphi}$. Thus we see that if we are
interested in linearization, and only at leading terms, we find that $\bar{V}
$ is precisely equal to our $V\left( \chi \right) $. A similar argument
gives the same result for closed $G_{2}$-structures as well \cite{BryantXu}.
This is not surprising, since to leading order, the evolution of the metric
is given by the Ricci flow in both cases.

\section{Modified flow}

\label{SecModFlow}\setcounter{equation}{0}In the previous section we have
seen that the original Laplacian flow of $\psi $ (\ref{orig-coflow}) is not
parabolic, and moreover, unlike the Laplacian flow of $\varphi $ (\ref%
{orig-flow}), it is not even weakly parabolic - as we have seen, the
principal symbol of $\Delta _{\psi }\psi $ is indefinite. Hence, unlike the
situation with the Ricci flow or the Laplacian flow of $\varphi $, we cannot
modify this flow to be parabolic just by adding a Lie derivative along a
vector $V$. In fact, in our case, we have seen that we choose $V$ such that
the $\Lambda _{1}^{4}\oplus \Lambda _{27}^{4}$ component becomes parabolic.
However it turns out that this addition cannot fix the highest order terms
of the $\Lambda _{7}^{4}$ component of the flow. The main motivation for
considering the Laplacian flow of $\psi $ is that it gives a flow of
co-closed $G_{2}$-structures with the volume functional increasing
monotonically along the flow, which suggests that if a long-time solution to
this flow were to exist, it would converge towards a torsion-free $G_{2}$%
-stucture. So let us modify the flow (\ref{orig-flow}) such that it becomes
weakly parabolic\ (before applying the DeTurck trick). However, it is
important that the new flow stays within the class of co-closed $G_{2}$%
-structures and that the volume functional increases along the flow.

\begin{theorem}
\label{thmmodcoflow}For a co-closed $G_{2}$-structure defined by the closed $%
4$-form $\psi $, consider the flow%
\begin{equation}
\frac{d\psi }{dt}=\Delta _{\psi }\psi +2d\left( \left( A-\func{Tr}T\right)
\varphi \right) .  \label{coflowmod2}
\end{equation}%
where $A$ is a fixed constant. Then, this is a weakly parabolic flow in the
direction of closed forms, and its linearization at a closed form $\chi $ is
given by 
\begin{equation*}
\frac{d\chi }{dt}=-\Delta _{\psi }\chi -\mathcal{L}_{V\left( \chi \right)
}\psi +d\hat{F}\left( \chi \right)
\end{equation*}%
where $V\left( \chi \right) $ is as in Proposition \ref{proplincoflow} and $%
\hat{F}\left( \chi \right) $ is a $3$-form-valued function that is algebraic
in $\chi $.

Moreover, the evolution of the volume functional (\ref{volfuncpsi}) is given
by%
\begin{equation}
\frac{dV}{dt}=\frac{1}{2}\int_{M}\left( \left\vert T\right\vert ^{2}+\func{Tr%
}T\left( 4A-3\func{Tr}T\right) \right) \mathrm{vol}.  \label{coflowmodvoldt}
\end{equation}
\end{theorem}

\begin{proof}
Let $\hat{Q}_{\psi }$ denote the operator on the right hand side of (\ref%
{coflowmod2}). Expand the extra term in $\hat{Q}_{\psi }\psi $:%
\begin{eqnarray}
2d\left( \left( A-\func{Tr}T\right) \varphi \right) &=&-2d\left( \func{Tr}%
T\right) \wedge \varphi +2\left( A-\func{Tr}T\right) d\varphi  \notag \\
&=&-2d\left( \func{Tr}T\right) \wedge \varphi +2\left( A-\func{Tr}T\right)
\left( 4\tau _{1}\psi -3\ast \mathrm{i}_{\varphi }\left( \tau _{27}\right)
\right)  \notag \\
&=&-2d\left( \func{Tr}T\right) \wedge \varphi +\frac{8}{7}\func{Tr}T\left( A-%
\func{Tr}T\right) \psi -6\left( A-\func{Tr}T\right) \ast \mathrm{i}_{\varphi
}\left( \tau _{27}\right)  \label{coflow2add}
\end{eqnarray}%
Thus the only highest order terms are present in the $\Lambda _{7}^{4}$
component. From (\ref{dttorspsi}) 
\begin{equation*}
D_{\psi }\left( \func{Tr}T\right) \left( \chi \right) =\func{div}X+l.o.t.
\end{equation*}%
So, from (\ref{coflowlinmod1}) the linearization of $\hat{Q}_{\psi }$ is now
given by 
\begin{eqnarray}
\pi _{7}\left( D_{\psi }\hat{Q}_{\psi }\right) \left( \chi \right)
&=&-d\left( \func{div}X\right) \wedge \varphi +l.o.t. \\
\pi _{1\oplus 27}\left( D_{\psi }\hat{Q}_{\psi }\right) \left( \chi \right)
&=&3\ast \mathrm{i}_{\varphi }\left( \nabla ^{2}h_{de}-\nabla
_{(d}^{{}}\left( \func{div}h\right) _{e)}-\left( \nabla _{a}^{{}}\nabla
_{(d}^{{}}X_{\left\vert b\right\vert }^{{}}\right) \varphi _{\ \
e)}^{ab}\right. \\
&&\left. +\frac{1}{4}\nabla _{d}\nabla _{e}\func{Tr}h-\frac{1}{4}\left(
\nabla ^{2}\func{Tr}h\right) g_{de}\right) +l.o.t.  \notag
\end{eqnarray}%
From Proposition \ref{proplincoflow} we get that 
\begin{equation*}
D_{\psi }\Delta _{\psi }\left( \chi \right) =-\Delta _{\psi }\chi -\mathcal{L%
}_{V\left( \chi \right) }\psi +2d\left( \left( \func{div}X\right) \varphi
\right) +dF\left( \chi \right)
\end{equation*}%
so, 
\begin{eqnarray*}
D_{\psi }\hat{Q}_{\psi }\left( \chi \right) &=&-\Delta _{\psi }\chi -%
\mathcal{L}_{V\left( \chi \right) }\psi +2d\left( \left( \func{div}X\right)
\varphi \right) +dF\left( \chi \right) \\
&&+2d\left( A\dot{\varphi}-\left( \func{div}X\right) \varphi \right) \\
&=&-\Delta _{\psi }\chi -\mathcal{L}_{V\left( \chi \right) }\psi +d\hat{F}%
\left( \chi \right)
\end{eqnarray*}%
where $\hat{F}=F+2A\dot{\varphi}$, so is also a $3$-form-valued function
that is algebraic in $\chi $. To see that this $\hat{Q}_{\psi }$ is weakly
parabolic, consider the alternative expression (\ref{pi127lappsilin3}) for
the linearization $\hat{Q}_{\psi }$ at a closed form $\chi $ 
\begin{eqnarray*}
\pi _{7}\left( D_{\psi }\hat{Q}_{\psi }\right) \left( \chi \right)
&=&-d\left( \func{div}X\right) \wedge \varphi +l.o.t. \\
&=&-\frac{1}{2}\left( \nabla _{a}\left( \func{div}X\right) +\nabla ^{2}X_{a}+%
\func{curl}\left( \func{div}h\right) _{a}\right) \wedge \varphi +l.o.t. \\
\pi _{1\oplus 27}\left( D_{\psi }\hat{Q}_{\psi }\right) \left( \chi \right)
&=&\frac{3}{2}\ast \mathrm{i}_{\varphi }\left( \nabla ^{2}h_{de}-\nabla
_{(d}^{{}}\left( \func{div}h\right) _{e)}-\nabla _{a}^{{}}\nabla
_{(d}^{{}}X_{\left\vert b\right\vert }^{{}}\varphi _{\ \ e)}^{ab}\right. \\
&&\left. -\left( \nabla _{b}^{{}}\nabla _{m}^{{}}h_{cn}^{{}}\right) \varphi
_{\ \ (d}^{bc}\varphi _{\ \ \ e)}^{mn}\right) +l.o.t.
\end{eqnarray*}%
where we have used the relation (\ref{pi7cond2}) to rewrite $\pi _{7}\left(
D_{\psi }\hat{Q}_{\psi }\right) $. From this, we can now write down the
principal symbol of $D_{\psi }\hat{Q}_{\psi }$ evaluated at some vector $\xi 
$:%
\begin{eqnarray*}
\pi _{7}\left( \sigma _{\xi }\left( D_{\psi }\hat{Q}_{\psi }\right) \chi
\right) &=&-\frac{1}{2}\left( \xi _{a}\xi _{b}X^{b}+\left\vert \xi
\right\vert ^{2}X_{a}+\xi _{m}^{{}}\xi _{{}}^{n}h_{np}^{{}}\varphi _{\ \ 
\text{\ }a}^{mp}\right) \wedge \varphi +l.o.t. \\
\pi _{1\oplus 27}\left( \sigma _{\xi }\left( D_{\psi }\hat{Q}_{\psi }\right)
\chi \right) &=&\frac{3}{2}\ast \mathrm{i}_{\varphi }\left( \left\vert \xi
\right\vert ^{2}h_{de}-\xi _{(d}^{{}}\xi _{\left\vert a\right\vert
}^{{}}h_{\ e)}^{a}-\xi _{a}^{{}}\xi _{(d}^{{}}X_{\left\vert b\right\vert
}^{{}}\varphi _{\ \ e)}^{ab}\right. \\
&&\left. -\xi _{b}^{{}}\xi _{m}^{{}}h_{cn}^{{}}\varphi _{\ \ (d}^{bc}\varphi
_{\ \ \ e)}^{mn}\right) +l.o.t.
\end{eqnarray*}%
Then, 
\begin{eqnarray*}
\left\langle \sigma _{\xi }\left( D_{\psi }\hat{Q}_{\psi }\right) \chi ,\chi
\right\rangle &=&2\left\vert \xi \right\vert ^{2}\left\vert X\right\vert
^{2}+2\left( \xi _{a}X^{a}\right) ^{2}+\xi _{m}^{{}}\xi
_{{}}^{n}h_{np}^{{}}\varphi _{\ \ \ a}^{mp}X_{{}}^{a} \\
&&+\left\vert \xi \right\vert ^{2}\left\vert h\right\vert ^{2}-\xi
_{d}^{{}}\xi _{a}^{{}}h_{\ e}^{a}h_{{}}^{de}-\xi _{b}\xi
_{m}h_{cn}h_{dp}\varphi _{\ \ \ }^{bcd}\varphi _{\ \ \ \ \ }^{mnp} \\
&=&\frac{7}{4}\left( \left\vert \xi \right\vert ^{2}\left\vert X\right\vert
^{2}+\left( \xi _{a}X^{a}\right) ^{2}\right) +2\omega _{ab}\omega ^{ab}\geq 0
\end{eqnarray*}%
where 
\begin{equation*}
\omega _{ab}=\xi _{m}\varphi _{\ \ \ \ \ (a}^{mn}h_{b)n}^{{}}+\frac{1}{2}\xi
_{(a}X_{b)}.
\end{equation*}%
Note that if $\chi \in \Lambda _{1}^{4}$, then $\left\langle \sigma _{\xi
}\left( D_{\psi }\hat{Q}_{\psi }\right) \chi ,\chi \right\rangle =0$, hence $%
\hat{Q}_{\psi }$ is indeed only weakly parabolic in the direction of closed
forms.

Now consider the evolution of the volume functional: 
\begin{equation*}
\frac{dV}{dt}=\int_{M}\frac{d}{dt}\sqrt{\det g}.
\end{equation*}%
From (\ref{cocloseddetgdt}) and applying (\ref{dtdetg}) to (\ref{coflow2add}%
), we have 
\begin{eqnarray*}
\frac{d}{dt}\sqrt{\det g} &=&\left( \frac{1}{2}\left( \left\vert
T\right\vert ^{2}+\left( \func{Tr}T\right) ^{2}\right) +2\func{Tr}T\left( A-%
\func{Tr}T\right) \right) \sqrt{\det g} \\
&=&\frac{1}{2}\left( \left\vert T\right\vert ^{2}+\func{Tr}T\left( 4A-3\func{%
Tr}T\right) \right) \sqrt{\det g}.
\end{eqnarray*}
\end{proof}

In particular we see that if at every time $t$ along the flow the following
holds 
\begin{equation}
A\int_{M}\func{Tr}T\mathrm{vol}\geq \frac{3}{4}\int_{M}\left( \func{Tr}%
T\right) ^{2}\mathrm{vol}  \label{trTcond}
\end{equation}%
then the volume functional grows along the flow (\ref{coflowmod2}). This is
of course true if for positive $A$ 
\begin{equation}
0\leq \func{Tr}T\leq \frac{4}{3}A  \label{trTcond2}
\end{equation}%
holds everywhere on $M$ for all $t$. So as long as the condition (\ref%
{trTcond}) holds, our flow satisfies all the desired properties. We can also
get an alternative condition for the positivity of $\frac{dV}{dt}$. From (%
\ref{coclosedRicci}) we have an expression for the scalar curvature $R$ of a
co-closed $G_{2}$-stucture.:%
\begin{equation}
R=-\left\vert T\right\vert ^{2}+\left( \func{Tr}T\right) ^{2}
\label{scalcurvcoclosed}
\end{equation}%
Hence, we can rewrite the evolution of $\sqrt{\det g}$ as 
\begin{equation*}
\frac{d}{dt}\sqrt{\det g}=\frac{1}{2}\left( -R+2\func{Tr}T\left( 2A-\func{Tr}%
T\right) \right) \sqrt{\det g}.
\end{equation*}%
Thus, $\frac{dV}{dt}\geq 0$ if and only if 
\begin{equation}
\int_{M}R\mathrm{vol}\leq 2\int_{M}\func{Tr}T\left( 2A-\func{Tr}T\right) 
\mathrm{vol}  \label{Rcond}
\end{equation}

In order to understand whether any of the conditions (\ref{trTcond}), (\ref%
{trTcond2}) or (\ref{Rcond}) have any hope of holding along the flow, we
first need to understand how the torsion evolves along the flow (\ref%
{coflowmod2}).

\begin{proposition}
The evolution of the torsion tensor $T$ under the flow (\ref{coflowmod2}) is
given by 
\begin{eqnarray}
\frac{d}{dt}T_{ab} &=&\Delta _{L}T_{ab}+\left( \func{Tr}T-2A\right) \left( 
\func{curl}T\right) _{ab}+4T_{{}}^{np}\left( \nabla _{n}^{{}}T_{\
(a}^{m}\right) \varphi _{b)pm}^{{}}  \label{Tabmodflow} \\
&&+2\left( \nabla _{(a}^{{}}T_{{}}^{mn}\right) \varphi _{b)mp}^{{}}T_{n}^{\
p}+2\left( \nabla ^{m}\func{Tr}T\right) \varphi _{mn(a}^{{}}T_{b)}^{\
n}+\left( \left( \func{curl}T\right) \circ T\right) _{ab}  \notag \\
&&+\frac{2}{3}\psi _{acde}\psi _{bmnp}T^{cm}T^{dn}T^{ep}+2\left(
T^{3}\right) _{ab}+2\left( A-\func{Tr}T\right) \left( T^{2}\right) _{ab}+%
\frac{1}{2}\left( \left( \func{Tr}T\right) ^{2}-3\left\vert T\right\vert
^{2}\right) T_{ab}  \notag \\
&&+g_{ab}\left( \frac{1}{2}\left( \func{Tr}T\right) \left\vert T\right\vert
^{2}-\frac{1}{3}\func{Tr}\left( T^{3}\right) +\frac{1}{6}\left\langle T\circ
T,T\right\rangle -\frac{1}{6}\left( \func{Tr}T\right) ^{3}\right)  \notag
\end{eqnarray}%
where $\Delta _{L}$ denotes the Lichnerowicz Laplacian, given by 
\begin{equation*}
\Delta _{L}T_{ab}=\nabla ^{2}T_{ab}-2R_{\ (b}^{e}T_{a)e}^{{}}+2R_{acbd}T^{cd}
\end{equation*}%
Also, the evolution of $\func{Tr}T$ under the same flow is given by 
\begin{equation}
\frac{d}{dt}\func{Tr}T=\nabla ^{2}\func{Tr}T-\left\langle \func{curl}%
T,T\right\rangle -\frac{1}{2}\left\langle T\circ T,T\right\rangle -\func{Tr}%
\left( T^{3}\right) -2\left( A-\func{Tr}T\right) \left\vert T\right\vert ^{2}
\label{TrTmodflow}
\end{equation}
\end{proposition}

\begin{proof}
From Proposition \ref{PropDtdtgen}, we know that if 
\begin{equation*}
\frac{d}{dt}\psi =\ast \left( X\lrcorner \psi \right) +3\ast \mathrm{i}%
_{\varphi }\left( h\right)
\end{equation*}%
then, 
\begin{equation*}
\frac{dT_{ab}}{dt}=\frac{1}{4}\left( \func{Tr}h\right) T_{ab}-T_{a}^{\
c}h_{cb}^{{}}-T_{a}^{\ c}X_{{}}^{d}\varphi _{dcb}^{{}}+\left( \func{curl}%
h\right) _{ab}+\nabla _{a}X_{b}-\frac{1}{4}\left( \nabla _{c}^{{}}\func{Tr}%
h\right) \varphi _{\ ab}^{c}.
\end{equation*}%
Then using (\ref{lapdecomsym}) and (\ref{dphicc}) we find that for the flow (%
\ref{coflowmod2}), we have 
\begin{subequations}%
\label{modflowcomps} 
\begin{eqnarray}
X &=&\nabla \func{Tr}T \\
h_{ab} &=&-\func{curl}\left( T\right) _{ab}-\frac{1}{2}\left( T\circ
T\right) _{ab}-\left( T^{2}\right) _{ab}+\frac{1}{6}g_{ad}\left\vert
T\right\vert ^{2}-2\left( A-\func{Tr}T\right) T_{ad} \\
&&+\frac{1}{6}\func{Tr}T\left( 4A-3\func{Tr}T\right) g_{ad} \\
\func{Tr}h &=&\frac{2}{3}\left( \left\vert T\right\vert ^{2}+\func{Tr}%
T\left( 4A-3\func{Tr}T\right) \right)  \notag
\end{eqnarray}%
\end{subequations}%
We then substitute $X,$ $h$ and $\func{Tr}h$ into the expression for $\frac{%
dT_{ab}}{dt}$, and simplify using the contraction identities (\ref{contids}%
), the Ricci identity and the expression for the Ricci curvature (\ref%
{coclosedRicci}). After a lengthy calculation, the expression (\ref%
{Tabmodflow}) follows.

The evolution of $\func{Tr}T$ can be derived from (\ref{Tabmodflow}), but an
easier way is to calculate it first for a general flow using Proposition (%
\ref{propflow}) and Proposition \ref{PropDtdtgen}, and then applying (\ref%
{modflowcomps}) 
\begin{eqnarray*}
\frac{d}{dt}\func{Tr}T &=&\frac{d}{dt}\left( g^{ab}T_{ab}\right) \\
&=&\left( \frac{d}{dt}g^{ab}\right) T_{ab}+g^{ab}\frac{d}{dt}\left(
T_{ab}\right) \\
&=&\left( -\frac{1}{2}\left( \func{Tr}h\right) g^{ab}+2h^{ab}\right) T_{ab}
\\
&&+g^{ab}\left( \frac{1}{4}\left( \func{Tr}h\right) T_{ab}-T_{a}^{\
c}h_{cb}^{{}}-T_{a}^{\ c}X_{{}}^{d}\varphi _{dcb}^{{}}+\left( \func{curl}%
h\right) _{ab}+\nabla _{a}X_{b}-\frac{1}{4}\left( \nabla _{c}^{{}}\func{Tr}%
h\right) \varphi _{\ ab}^{c}\right) \\
&=&-\frac{1}{4}\func{Tr}h\func{Tr}T+\left\langle T,h\right\rangle +\func{div}%
X \\
&=&\nabla ^{2}\func{Tr}T-\left\langle \func{curl}\left( T\right)
,T\right\rangle -\frac{1}{2}\left\langle T\circ T,T\right\rangle -\func{Tr}%
\left( T^{3}\right) +\frac{1}{6}\left\vert T\right\vert ^{2}\func{Tr}%
T-2\left( A-\func{Tr}T\right) \left\vert T\right\vert ^{2} \\
&&+\frac{1}{6}\left( \func{Tr}T\right) ^{2}\left( 4A-3\func{Tr}T\right) -%
\frac{1}{6}\left( \left\vert T\right\vert ^{2}+\func{Tr}T\left( 4A-3\func{Tr}%
T\right) \right) \\
&=&\nabla ^{2}\func{Tr}T-\left\langle \func{curl}T,T\right\rangle -\frac{1}{2%
}\left\langle T\circ T,T\right\rangle -\func{Tr}\left( T^{3}\right) -2\left(
A-\func{Tr}T\right) \left\vert T\right\vert ^{2}
\end{eqnarray*}
\end{proof}

In (\ref{TrTmodflow}), if \ $A$ is large enough and if $T_{ab}$ is positive
definite, then at least the non-derivative terms will be negative. This is
due to the following observation regarding the $G_{2}$ product $\circ $.

\begin{lemma}
Suppose $A$ and $B$ are positive-definite symmetric $2$-tensors. Then the
product $A\circ B$ is also positive definite.
\end{lemma}

\begin{proof}
The matrices $A$ and $B$ have unique positive-definite symmetric square root
matrices $A^{\frac{1}{2}}$ and $B^{\frac{1}{2}}$. So we can rewrite $A\circ
B $ as 
\begin{eqnarray*}
\left( A\circ B\right) _{ab} &=&\varphi _{amn}\varphi _{bpq}\left( A^{\frac{1%
}{2}}\right) ^{mc}\left( A^{\frac{1}{2}}\right) _{c}^{\ p}\left( B^{\frac{1}{%
2}}\right) ^{nd}\left( B^{\frac{1}{2}}\right) _{d}^{\ q} \\
&=&U_{a}U_{b}
\end{eqnarray*}%
where 
\begin{equation*}
U_{a}=\left( \varphi _{amn}\left( A^{\frac{1}{2}}\right) ^{mc}\left( B^{%
\frac{1}{2}}\right) ^{nd}\right) .
\end{equation*}%
Hence $A\circ B$ is indeed positive definite.
\end{proof}

However in order to be able to use the Maximum Principle to conclude that $%
\func{Tr}T$ is always bounded from above, if it is initially so, we would
need to show that a positive definite $T_{ab}$ remains positive definite
along the flow and we also need to be able to control the term $\left\langle 
\func{curl}T,T\right\rangle $ in (\ref{TrTmodflow}). Equivalently we would
need to be able to control $\left\langle Ric,T\right\rangle $, since $\func{%
curl}T$ enters the expression for the Ricci curvature (\ref{coclosedRicci}).
This is certainly not obvious from (\ref{Tabmodflow}), however given the
similarities with Ricci flow, and corresponding results about the evolution
of the Ricci and scalar curvatures along the Ricci flow, it seems to be a
reasonable conjecture that $T_{ab}$ is positive and $\func{Tr}T$ does in
fact satisfy (\ref{trTcond2}) along the flow. The properties of the torsion
and the curvature along the flow (\ref{coflowmod2}) will be subject to
further research.

\section{Short-time existence and uniqueness}

\label{SecShortTime}\setcounter{equation}{0}Let us now formulate the initial
value problem for which we want to prove short-time existence and uniqueness:

\begin{equation}
\left\{ 
\begin{array}{c}
\frac{d}{dt}\psi =\Delta _{\psi }\psi +2d\left( \left( A-\func{Tr}T_{\psi
}\right) \ast _{\psi }\psi \right) \\ 
d\psi =0 \\ 
\left. \psi \right\vert _{t=0}=\psi _{0}%
\end{array}%
\right.  \label{ivp1}
\end{equation}%
Here will adapt the method that was used by Bryant and Xu \cite{BryantXu} to
show short-time existence and uniqueness for the Laplacian flow of $\varphi $
(\ref{orig-flow}). Since for closed $\psi \left( t\right) $, the right hand
side of the flow is exact, for any $t,$ $\psi \left( t\right) $ stays within
the same cohomology class $\left[ \psi _{0}\right] $ and hence we can write 
\begin{equation}
\psi \left( t\right) =\psi _{0}+\chi \left( t\right)
\end{equation}%
where $\chi \left( t\right) $ is an exact form. We can thus rewrite (\ref%
{ivp1}) in an equivalent way:%
\begin{equation}
\left\{ 
\begin{array}{c}
\frac{d}{dt}\chi =\Delta _{\psi }\psi +2d\left( \left( A-\func{Tr}T_{\psi
}\right) \ast _{\psi }\psi \right) \\ 
\chi \ \ \text{is exact} \\ 
\left. \chi \right\vert _{t=0}=0%
\end{array}%
\right.  \label{ivp2}
\end{equation}%
As before, suppose $\chi $ is given by 
\begin{equation*}
\chi =\ast \left( X\lrcorner \psi +3\mathrm{i}_{\varphi }\left( h\right)
\right) .
\end{equation*}%
Then, as before, define the vector 
\begin{equation}
V\left( \chi \right) =\frac{3}{4}\nabla \func{Tr}h-2\func{curl}X
\label{Vchi}
\end{equation}%
From Theorem \ref{thmmodcoflow}, we can say that the following initial value
problem is parabolic in the direction of closed forms%
\begin{equation}
\left\{ 
\begin{array}{c}
\frac{d}{dt}\chi =\Delta _{\psi }\psi +2d\left( \left( A-\func{Tr}T_{\psi
}\right) \ast _{\psi }\psi \right) +\mathcal{L}_{V\left( \chi \right) }\psi
\\ 
\chi \ \ \text{is exact} \\ 
\left. \chi \right\vert _{t=0}=0%
\end{array}%
\right.  \label{ivp3}
\end{equation}%
\newline
The idea is to first show short-time existence and uniqueness of solutions
for the flow (\ref{ivp3}), and then from these, obtain solutions of (\ref%
{ivp2}) via diffeomorphisms. Since (\ref{ivp3}) is parabolic only in certain
directions, standard theory of parabolic PDEs does not apply, and we thus
have to use the Nash-Moser inverse function theorem for tame Fr\'{e}chet
spaces. This technique was first used by Hamilton for the Ricci flow \cite%
{Hamilton3folds,HamiltonNashMoser} and Bryant and Xu for the Laplacian flow
of $\varphi $ \cite{BryantXu}. We first review basic definitions as
introduced by Hamilton \cite{HamiltonNashMoser}.

\begin{definition}
\begin{enumerate}
\item A \emph{graded Fr\'{e}chet space }$\mathcal{F}$ is a complete
Hausdorff topological vector space with the topology defined by a collection
of increasing semi-norms $\left\{ \left\Vert \cdot \right\Vert _{n}\right\}
_{n=1}^{\infty }$, so that%
\begin{equation*}
\left\Vert x\right\Vert _{0}\leq \left\Vert x\right\Vert _{1}\leq \left\Vert
x\right\Vert _{2}\leq ...
\end{equation*}%
for every $x\in \mathcal{F},$ so that a sequence converges if and only if it
converges with respect to each semi-norm and a subset $\mathcal{U\subset }$ $%
\mathcal{F}$ is open if and only if around every point $x\in \mathcal{U}$
there exists an open ball contained in $\mathcal{U}$ with respect to the
semi-norm $\left\Vert \cdot \right\Vert _{n}$ for some $n.$ Such a
collection of norms is called a \emph{grading}.

\item Two gradings $\left\{ \left\Vert \cdot \right\Vert _{n}\right\}
_{n=1}^{\infty }$ and $\left\{ \left\Vert \cdot \right\Vert _{n}^{\prime
}\right\} _{n=1}^{\infty }$ are said to be \emph{tamely equivalent }of
degree $r$ and base $b$, if for all $x\in \mathcal{F}$ 
\begin{equation}
\left\Vert x\right\Vert _{n}\leq C\left( n\right) \left\Vert x\right\Vert
_{n+r}^{\prime }\ \ \text{and }\left\Vert x\right\Vert _{n}^{\prime }\leq
C\left( n\right) \left\Vert x\right\Vert _{n+r}
\end{equation}%
for all $n\geq b$.

\item Let $\mathcal{F}$ and $\mathcal{G}$ be graded Fr\'{e}chet spaces. Then
a linear map $L:\mathcal{F}$ $\longrightarrow $ $\mathcal{G}$ is \emph{tame }%
if it satisfies a \emph{tame estimate }of degree $r$ and base $b,$ that is,
for each $n\geq b$ the following holds:\emph{\ }%
\begin{equation}
\left\Vert Lx\right\Vert _{n}\leq C\left( n\right) \left\Vert x\right\Vert
_{n+r}
\end{equation}%
for some constant $C\left( n\right) $ that may depend on $n$. A tame linear
map is continuous in the topology of $\mathcal{F}$.

\item Suppose $P:\mathcal{F}$ $\longrightarrow $ $\mathcal{G}$ is a
continuous map. Then $P$ is tame if it satisfies the following tame estimate
of degree $r$ and base $b$:%
\begin{equation}
\left\Vert Px\right\Vert _{n}\leq C\left( n\right) \left( 1+\left\Vert
x\right\Vert _{n+r}\right) \ \text{for each }n\geq b
\end{equation}%
A tame map is said to be \emph{smooth tame }if all its derivatives are tame.

\item Let $\left( B,\left\Vert \cdot \right\Vert _{B}\right) $ be a Banach
space, then $\Sigma \left( B\right) $ denotes the graded Fr\'{e}chet space
of all sequences $\left\{ x_{k}\right\} _{k\in \mathbb{N}}$ in $B$ such that
for all $n\geq 0$ 
\begin{equation}
\left\Vert \left\{ x_{k}\right\} _{k\in \mathbb{N}}\right\Vert
_{n}:=\sum_{k=0}^{\infty }e^{nk}\left\Vert x_{k}\right\Vert _{B}<\infty
\end{equation}

\item A graded Fr\'{e}chet space $\mathcal{F}$ is \emph{tame }if there
exists a Banach space $B$ and two tame linear maps $L:\mathcal{F}%
\longrightarrow \Sigma \left( B\right) $ and $M:\Sigma \left( B\right)
\longrightarrow \mathcal{F}$ such that $M\circ L$ is the identity on $%
\mathcal{F}$.
\end{enumerate}
\end{definition}

The main reason for introducing the Fr\'{e}chet space formalism is the
Nash-Moser inverse function theorem.

\begin{theorem}[Nash-Moser Inverse Function Theorem \protect\cite%
{HamiltonNashMoser}]
\label{ThmNashMoser}Let $\mathcal{F}$ and $\mathcal{G}$ be tame Fr\'{e}chet
spaces, and $f:\mathcal{U}\subset \mathcal{F}$ $\longrightarrow $ $\mathcal{G%
}$ a smooth tame map. Suppose that

\begin{enumerate}
\item The derivative $Df\left( x\right) :\mathcal{F}$ $\longrightarrow $ $%
\mathcal{G}$ is a linear isomorphism for all $x\in \mathcal{U}$

\item The map $\mathcal{U}\times $ $\mathcal{G}\longrightarrow \mathcal{F}$
given by 
\begin{equation*}
\left( x,v\right) \longrightarrow \left( Df\left( x\right) \right)
^{-1}\left( v\right)
\end{equation*}%
is a smooth tame map. Then, $f$ is locally invertible and each local inverse 
$f^{-1}$ is a smooth tame map.
\end{enumerate}
\end{theorem}

As shown in \cite{HamiltonNashMoser}, if $M$ is a compact manifold, and $V$
is a vector bundle over $M$, then the space $C^{\infty }\left( M,V\right) $
of smooth sections of $V$ is in fact a tame space. Different tamely
equivalent gradings can be chosen, but usually either Sobolev or supremum
norms are used. Moreover, suppose $P$ is a non-linear (smooth) vector bundle
differential operator, then it is also shown in \cite{HamiltonNashMoser}
that it is in fact a smooth tame map. Hence Theorem \ref{ThmNashMoser} is
very useful for the study of non-linear PDEs since it provides sufficient
conditions for local existence of solutions.

Consider now tame spaces of time-dependent sections of vector bundles, as
introduced in \cite{Hamilton3folds}. Let $u\in C^{\infty }\left( \left[ 0,T%
\right] \times M,V\right) $ be a time dependent section of the vector bundle 
$V$ over $M$. Define 
\begin{equation}
\left\vert u\right\vert _{n}^{2}=\int_{0}^{T}\left\vert u\left( t\right)
\right\vert _{H^{n}}^{2}dt
\end{equation}%
where $\left\vert \cdot \right\vert _{H^{n}}$ is the Sobolev $L_{2}$ norm of 
$u$ and its covariant derivatives up to degree $n$. Hence $\left\vert
u\right\vert _{n}$ only takes into account space derivatives. Now, define 
\begin{equation}
\left\Vert u\right\Vert _{n}=\sum_{2j\leq n}\left\vert \left( \frac{\partial 
}{\partial t}\right) ^{j}u\right\vert _{n-2j}.  \label{timegrading}
\end{equation}%
This is a weighted norm counting one time derivative equal to two space
derivatives. In particular, the grading (\ref{timegrading}) makes $C^{\infty
}\left( \left[ 0,T\right] \times M,V\right) $ a tame space. We can also
define a tamely equivalent grading $\left\vert \left[ \cdot \right]
\right\vert _{n}$ given by 
\begin{equation}
\left\vert \left[ u\right] \right\vert _{n}=\sum_{2j\leq n}\left[ \left( 
\frac{\partial }{\partial t}\right) ^{j}u\right] _{n-2j}
\label{timegrading2}
\end{equation}%
where $\left[ u\right] _{n}$ is the supremum norm of $u$ and its space
derivatives up to degree $n$.

Since we are interested in the flow of exact forms (\ref{ivp3}), we
introduce the set 
\begin{equation}
\mathcal{U}=\left\{ \chi \in dC^{\infty }\left( \left[ 0,T\right] \times
M,\Lambda ^{3}\left( M\right) \right) :\psi _{0}+\chi \ \ \text{is a
definite }4\text{-form}\right\} .  \label{Uspace}
\end{equation}%
This is an open set of the space $\mathcal{F}$ of time-dependent exact $4$%
-forms on $M$: 
\begin{equation}
\mathcal{F}=dC^{\infty }\left( \left[ 0,T\right] \times M,\Lambda ^{3}\left(
M\right) \right)  \label{FrechetF}
\end{equation}%
Also define the space $\mathcal{G}$ of time-independent exact $4$-forms: 
\begin{equation}
\mathcal{G}=dC^{\infty }\left( M,\Lambda ^{3}\left( M\right) \right)
\label{FrechetG}
\end{equation}%
Now let 
\begin{equation*}
\mathcal{H}=\mathcal{F}\times \mathcal{G}
\end{equation*}%
and consider the map 
\begin{equation*}
F:\mathcal{U}\longrightarrow \mathcal{H}
\end{equation*}%
given by 
\begin{equation}
\chi \longrightarrow \left( \frac{d}{dt}\chi -\Delta _{\psi }\psi -2d\left(
\left( A-\func{Tr}T_{\psi }\right) \ast _{\psi }\psi \right) -\mathcal{L}%
_{V\left( \chi \right) }\psi ,\left. \chi \right\vert _{t=0}\right) .
\label{mapF}
\end{equation}%
Here $\chi =\psi -\psi _{0}$. Note that in \cite{BryantXu} exactly the same
spaces were considered for exact $3$-forms rather than $4$-forms, and the
map involved the Laplacian flow operator for the $3$-form $\varphi $. In 
\cite[Proposition 4.2]{BryantXu} it was shown that the spaces $dC^{\infty
}\left( \left[ 0,T\right] \times M,\Lambda ^{2}\left( M\right) \right) $, $%
dC^{\infty }\left( M,\Lambda ^{2}\left( M\right) \right) $ and hence their
product are tame. However exactly the same proof applies to exact $3$-forms.
Thus we have

\begin{proposition}[\protect\cite{BryantXu}]
The space $\mathcal{F}$ is a tame Fr\'{e}chet space with the grading $%
\left\Vert \cdot \right\Vert _{n}$ restricted from $C^{\infty }\left( \left[
0,T\right] \times M,\Lambda ^{3}\left( M\right) \right) $. The space $%
\mathcal{G}$ is a tame Fr\'{e}chet space with the grading $\left\vert \cdot
\right\vert _{n}$ and the product $\mathcal{H\,}=\mathcal{F}\times \mathcal{G%
}$ is a tame Fr\'{e}chet space with the grading $\left\Vert \cdot
\right\Vert _{n}+\left\vert \cdot \right\vert _{n}$.
\end{proposition}

The map $F$ is a smooth differential operator, so it is a tame map. Now to
apply the Nash-Moser Theorem to the map \ $F$ (\ref{mapF}) we thus need to
show that the derivative $DF\left( \chi \right) :\mathcal{F}\longrightarrow 
\mathcal{H}$ is an isomorphism for all $\chi \in \mathcal{U}$, and moreover
that its inverse is smooth tame. Again, the proofs of these facts are almost
exactly the same as in \cite[Proposition 4.2]{BryantXu}.

\begin{lemma}[{\protect\cite[Lemmas 4.3 and 4.4]{BryantXu}}]
\label{lemmapfiso}Given the map $F$ (\ref{mapF}), its derivative 
\begin{equation*}
DF\left( \chi \right) :\mathcal{F}\longrightarrow \mathcal{H}
\end{equation*}%
is an isomorphism for all $\chi \in \mathcal{U}.$
\end{lemma}

\begin{proof}
From Theorem \ref{thmmodcoflow}, we get that for closed $\theta $, the
derivative of the map $F$ is given by 
\begin{equation}
DF\left( \chi \right) \theta =\left( \frac{d}{dt}\theta +\Delta _{\psi
}\theta +d\hat{F}\left( \theta \right) ,\left. \theta \right\vert
_{t=0}\right)  \label{DFmap}
\end{equation}%
where $\psi =\psi _{0}+\chi $. To show injectivity we need to show that the
initial value problem for a closed $4$-form $\theta $ 
\begin{equation}
\left\{ 
\begin{array}{c}
\frac{d}{dt}\theta +\Delta _{\psi }\theta +d\hat{F}\left( \theta \right) =0
\\ 
\left. \theta \right\vert _{t=0}=0%
\end{array}%
\right.
\end{equation}%
has a unique solution $\theta =0.$ However this is a linear parabolic PDE,
and thus $\theta =0$ is the unique solution.

To show surjectivity, we have to show that for any time-dependent exact $4$%
-form $\eta $ and any exact $4$-form $\theta _{0}$ there is a time-dependent 
$3$-form $\theta $ which satisfies%
\begin{equation}
\left\{ 
\begin{array}{c}
\frac{d}{dt}\theta +\Delta _{\psi }\theta +d\hat{F}\left( \theta \right)
=\eta \\ 
\left. \theta \right\vert _{t=0}=\theta _{0}%
\end{array}%
\right.  \label{DFsurjivp}
\end{equation}%
Since $\theta $ and $\eta $ are all exact, we can write 
\begin{eqnarray*}
\theta &=&d\alpha \\
\eta &=&d\beta
\end{eqnarray*}%
for time-dependent $3$-forms $\alpha $ and $\beta $. Correspondingly, $%
\theta _{0}=d\beta _{0}$. Now consider the following initial value problem:%
\begin{equation}
\left\{ 
\begin{array}{c}
\frac{d}{dt}\alpha +\Delta _{\psi }\alpha +\hat{F}\left( d\alpha \right)
=\beta \\ 
\left. \beta \right\vert _{t=0}=\beta _{0}%
\end{array}%
\right.
\end{equation}%
From standard parabolic theory, there exists a unique solution $\beta \left(
t\right) $. Hence $\theta =d\beta $ solves (\ref{DFsurjivp}).
\end{proof}

Now that we have that $DF\left( \chi \right) $ is an isomorphism for each $%
\chi \in \mathcal{U}$, we can define the family of inverse maps 
\begin{eqnarray*}
\left( DF\right) ^{-1} &:&\mathcal{U}\times \mathcal{H}\longrightarrow 
\mathcal{F} \\
\left( \chi ,\theta ,\theta _{0}\right) &\longrightarrow &\left( DF\left(
\chi \right) \right) ^{-1}\left( \theta ,\theta _{0}\right)
\end{eqnarray*}

\begin{lemma}
The map $\left( DF\right) ^{-1}:\mathcal{U}\times \mathcal{H}\longrightarrow 
\mathcal{F}\ $is smooth tame.
\end{lemma}

\begin{proof}
From (\ref{lemmapfiso}) we know that $\theta =\left( DF\right) ^{-1}\left(
\chi ,\eta ,\theta _{0}\right) $ is the unique solution to the linear
parabolic equation%
\begin{equation}
\left\{ 
\begin{array}{c}
\frac{d}{dt}\theta +\Delta _{\psi }\theta +d\hat{F}\left( \theta \right)
=\eta \\ 
\left. \theta \right\vert _{t=0}=\theta _{0}%
\end{array}%
\right.
\end{equation}%
with $\psi =\psi _{0}+\chi $. It then follows from standard linear parabolic
theory that $\left( DF\right) ^{-1}$ is smooth tame. For details we refer
the reader to \cite[Lemma 4.7]{BryantXu}, where exactly the same result is
proven for $3$-forms. It applies in the same way to $4$-forms.
\end{proof}

Now we can apply the Nash Moser theorem to the map $F$.

\begin{lemma}[{\protect\cite[Lemma 0.2]{BryantXu}}]
\label{lem02}Suppose $\chi \in \mathcal{U}$ is a solution to 
\begin{equation}
\left\{ 
\begin{array}{c}
\frac{d}{dt}\chi -\Delta _{\psi }\psi -2d\left( \left( A-\func{Tr}T_{\psi
}\right) \ast _{\psi }\psi \right) -\mathcal{L}_{V\left( \chi \right) }\psi
=\eta \\ 
\left. \chi \right\vert _{t=0}=\chi _{0}%
\end{array}%
\right.
\end{equation}%
Then for $\left( \bar{\eta},\bar{\chi}_{0}\right) \in \mathcal{H}$
sufficiently close to $\left( \eta ,\chi _{0}\right) $, there is a unique
solution $\bar{\chi}\left( t\right) $ to the system 
\begin{equation}
\left\{ 
\begin{array}{c}
\frac{d}{dt}\bar{\chi}-\Delta _{\bar{\psi}}\bar{\psi}-2d\left( \left( A-%
\func{Tr}T_{\bar{\psi}}\right) \ast _{\bar{\psi}}\bar{\psi}\right) -\mathcal{%
L}_{V\left( \bar{\chi}\right) }\bar{\psi}=\bar{\eta} \\ 
\left. \bar{\chi}\right\vert _{t=0}=\bar{\chi}_{0}%
\end{array}%
\right.  \label{chibarsys}
\end{equation}%
where $\bar{\psi}=\psi _{0}+\bar{\chi}$.
\end{lemma}

\begin{proof}
The derivative $DF$ of the map $F$ (\ref{mapF}) satisfies the conditions of
Theorem \ref{ThmNashMoser}, and hence $F$ itself is locally invertible.
Therefore, given $\left( \bar{\eta},\bar{\chi}_{0}\right) \in \mathcal{H}$
sufficiently close to $\left( \eta ,\chi _{0}\right) $, there exists a
solution $\bar{\chi}=F^{-1}\left( \bar{\eta},\bar{\chi}_{0}\right) $.
\end{proof}

It is now straightforward to obtain short-time existence and uniqueness of
solutions to the gauge-fixed flow (\ref{ivp3}).

\begin{corollary}
\label{corrivp3sol}The initial value problem (\ref{ivp3}) has a unique
solution for the time period $\left[ 0,\varepsilon \right] $ for some $%
\varepsilon >0$.
\end{corollary}

\begin{proof}
Here we apply the same method as used in \cite{BryantXu}, and originally by
Hamilton in \cite{Hamilton3folds}. Let $\chi \left( t\right) $ be a family
of $4$-forms such that its formal Taylor series at $t=0$ is what it must be
to solve (\ref{ivp3}) with $\chi \left( 0\right) =\chi _{0}$. This can be
done by differentiating through the flow equation (\ref{ivp3}) and solving
for $\left. \frac{d^{k}\chi }{dt^{k}}\right\vert _{t=0}$. Then let 
\begin{equation}
\eta \left( t\right) =\frac{d}{dt}\chi -\Delta _{\psi }\psi -2d\left( \left(
A-\func{Tr}T_{\psi }\right) \ast _{\psi }\psi \right) -\mathcal{L}_{V\left(
\chi \right) }\psi
\end{equation}%
where $\psi =\psi _{0}+\chi $. It follows that the formal Taylor series of $%
\eta $ at $t=0$ is identically zero. We then extend $\eta \left( t\right) $
such that $\eta \left( t\right) =0$ for $t<0$. Now define $\bar{\eta}$ to be
the translation of $\eta $ by some $\varepsilon >0$%
\begin{equation*}
\bar{\eta}\left( t\right) =\eta \left( t-\varepsilon \right) .
\end{equation*}%
Thus, we get a $4$-form $\bar{\eta}\left( t\right) $ which vanishes for $%
t\in \left[ 0,\varepsilon \right] $ for some $\varepsilon >0$. We can then
apply Lemma \ref{lem02} to the two pairs $\left( \bar{\eta},\chi _{0}\right) 
$ and $\left( \eta ,\chi _{0}\right) $, possibly for a shorter time period $%
\left[ 0,\varepsilon ^{\prime }\right] $. This then gives existence of a
unique solution $\bar{\chi}\left( t\right) $ to the system (\ref{chibarsys}%
). However, since $\bar{\eta}$ vanishes up to time $\varepsilon ^{\prime }$,
we get a unique solution to (\ref{ivp3}) for $t\in \left[ 0,\varepsilon
^{\prime }\right] $.
\end{proof}

Following \cite{BryantXu}, to prove short-time existence and uniqueness for
the initial value problem (\ref{ivp2}), we need to relate (\ref{ivp3}) and (%
\ref{ivp2}) via diffeomorphisms. Let $\psi \left( t\right) $ be a family of $%
G_{2}$-structure $4$-forms, and $\phi \left( t\right) $ a family of
diffeomorphisms defined by the evolution equation 
\begin{equation}
\left\{ 
\begin{array}{c}
\frac{d}{dt}\phi \left( t\right) =-V\left( \bar{\chi}_{\phi }\left( t\right)
\right) \\ 
\left. \phi \right\vert _{t=0}=Id%
\end{array}%
\right.  \label{ddtdiffeo}
\end{equation}%
where $\bar{\chi}_{\phi }\left( t\right) =\left( \phi ^{-1}\right) ^{\ast
}\psi -\psi _{0}$ and $V$ is given by (\ref{Vchi}).

\begin{lemma}
\label{diffeoflowpara}The flow (\ref{ddtdiffeo}) is strictly parabolic.
\end{lemma}

\begin{proof}
We have to linearize (\ref{ddtdiffeo}). Suppose 
\begin{equation}
\left. \frac{\partial \phi }{\partial s}\right\vert _{s=0}=U
\end{equation}%
for some vector field $U$. Then, 
\begin{equation}
\left. \frac{\partial \phi ^{-1}}{\partial s}\right\vert _{s=0}=-\left( \phi
^{-1}\right) _{\ast }U
\end{equation}%
and hence, 
\begin{eqnarray}
\left. \frac{\partial \left( \phi ^{-1}\right) ^{\ast }\psi }{\partial s}%
\right\vert _{s=0} &=&\left( \phi ^{-1}\right) _{\ast }\mathcal{L}_{-\left(
\phi ^{-1}\right) _{\ast }U}\psi  \notag \\
&=&-\mathcal{L}_{U}\left( \phi ^{-1}\right) ^{\ast }\psi
\end{eqnarray}%
Now, define $\bar{\psi}=\left( \phi ^{-1}\right) ^{\ast }\psi $. This
defines now a new $G_{2}$-structure. So we have 
\begin{equation}
\left. \frac{\partial \bar{\psi}}{\partial s}\right\vert _{s=0}=-\mathcal{L}%
_{U}\bar{\psi}
\end{equation}%
So on the right hand side of (\ref{ddtdiffeo}), $\bar{\chi}=\bar{\psi}-\psi
_{0}$ and whence, 
\begin{equation}
-V\left( \bar{\chi}\right) =-\left( \frac{3}{4}\bar{\nabla}^{a}\func{Tr}%
h+\left( \overline{\func{curl}}X\right) ^{a}\right)
\end{equation}%
where everything is with respect to the $G_{2}$-structure $\bar{\psi}$, and $%
\bar{\chi}$ is given by%
\begin{equation*}
\bar{\chi}=\ast \left( X\lrcorner \bar{\psi}+3\mathrm{i}_{\bar{\varphi}%
}\left( h\right) \right) .
\end{equation*}%
Hence 
\begin{equation*}
\left. -\frac{\partial V\left( \bar{\chi}\right) }{\partial s}\right\vert
_{s=0}=-\left. \frac{\partial }{\partial s}\left( \frac{3}{4}\bar{\nabla}^{a}%
\func{Tr}h+\left( \overline{\func{curl}}X\right) ^{a}\right) \right\vert
_{s=0}
\end{equation*}%
Now, 
\begin{eqnarray*}
-\mathcal{L}_{U}\bar{\psi} &=&-d\left( U\lrcorner \bar{\psi}\right)
-U\lrcorner d\bar{\psi}. \\
&=&-d\left( U\lrcorner \bar{\psi}\right)
\end{eqnarray*}%
since $\bar{\psi}$ is closed, so from Proposition \ref{PropComponents}, we
get the type decomposition of $-\mathcal{L}_{U}\bar{\psi}$%
\begin{eqnarray}
\pi _{7}\left( -\mathcal{L}_{U}\bar{\psi}\right) &=&\frac{1}{2}\left( 
\overline{\func{curl}}U\right) \wedge \bar{\varphi}+l.o.t. \\
\pi _{1\oplus 27}\left( -\mathcal{L}_{U}\bar{\psi}\right) &=&3\ast \mathrm{i}%
_{\varphi }\left( \bar{\nabla}_{(m}U_{n)}-\frac{1}{3}\left( \overline{\func{%
div}}U\right) \bar{g}_{mn}\right) +l.o.t.
\end{eqnarray}%
Since 
\begin{equation*}
\left. \frac{\partial \bar{\chi}}{\partial s}\right\vert _{s=0}=\left. \frac{%
\partial \bar{\psi}}{\partial s}\right\vert _{s=0}=-\mathcal{L}_{U}\bar{\psi}
\end{equation*}%
we find that 
\begin{eqnarray*}
\left. \frac{\partial X_{a}}{\partial s}\right\vert _{s=0} &=&\frac{1}{2}%
\left( \overline{\func{curl}}U\right) _{a}+l.o.t. \\
\left. \frac{\partial h_{mn}}{\partial s}\right\vert _{s=0} &=&\bar{\nabla}%
_{(m}U_{n)}-\frac{1}{3}\left( \overline{\func{div}}U\right) \bar{g}%
_{mn}+l.o.t.
\end{eqnarray*}%
and therefore, 
\begin{eqnarray*}
\left. \frac{\partial \left( \bar{\nabla}^{a}\func{Tr}h\right) }{\partial s}%
\right\vert _{s=0} &=&-\frac{4}{3}\bar{\nabla}^{a}\overline{\func{div}}%
U+l.o.t. \\
\text{ }\left. \frac{\partial \left( \bar{\nabla}_{m}X_{n}\bar{\varphi}_{\ \
\ }^{mna}\right) }{\partial s}\right\vert _{s=0} &=&\frac{1}{2}\left( \bar{%
\nabla}_{m}^{{}}\bar{\nabla}_{p}^{{}}U_{q}^{{}}\right) \bar{\varphi}_{\ \
n}^{pq}\bar{\varphi}_{{}}^{mna}+l.o.t. \\
&=&-\frac{1}{2}\nabla ^{2}U^{a}+\frac{1}{2}\bar{\nabla}^{a}\overline{\func{%
div}}U+l.o.t.
\end{eqnarray*}%
where in the last line we have used once again the identity (\ref{phiphi1}%
).Whence, the linearization \ $D_{\phi }V$ of $V$ at $\phi $ is given by 
\begin{equation}
\left( D_{\phi }V\right) \left( U\right) =\frac{1}{2}\nabla ^{2}U^{a}+\frac{1%
}{2}\bar{\nabla}^{a}\overline{\func{div}}U+l.o.t.
\end{equation}%
Given some non-zero vector field $\xi $, the principal symbol is now 
\begin{equation*}
\sigma _{\xi }\left( D_{\phi }V\right) \left( U\right) =\frac{1}{2}%
\left\vert \xi \right\vert ^{2}U^{a}+\frac{1}{2}\xi ^{a}\left\langle \xi
,U\right\rangle
\end{equation*}%
and hence 
\begin{equation*}
\left\langle \sigma _{\xi }\left( D_{\phi }V\right) \left( U\right)
,U\right\rangle =\frac{1}{2}\left\vert \xi \right\vert ^{2}\left\vert
U\right\vert ^{2}+\frac{1}{2}\left\langle \xi ,U\right\rangle ^{2}\geq 0
\end{equation*}%
with equality if and only if $U$ is identically zero. So indeed, the flow (%
\ref{ddtdiffeo}) is indeed strictly parabolic.
\end{proof}

Now are finally ready to prove the short-time existence and uniqueness for
the initial value problem (\ref{ivp1}). The proof is again modelled on \cite%
{BryantXu}.

\begin{theorem}
The initial value problem (\ref{ivp1}) has a unique solution for $t\in \left[
0,\varepsilon \right] $.
\end{theorem}

\begin{proof}
From Corollary \ref{corrivp3sol}, we have a short-time solution $\bar{\chi}=%
\bar{\psi}-\psi _{0}$ for (\ref{ivp3}). In (\ref{ddtdiffeo}) let $\bar{\chi}%
_{\phi }\left( t\right) =\bar{\chi}\left( t\right) ,$ so that it now becomes
an ODE, since $\bar{\chi}\left( t\right) $ is now a fixed family of $4$%
-forms, independent of the diffeomorphisms $\phi _{t}$:%
\begin{equation}
\left\{ 
\begin{array}{c}
\frac{d}{dt}\phi \left( t\right) =-V\left( \bar{\chi}\left( t\right) \right)
\\ 
\left. \phi \right\vert _{t=0}=Id%
\end{array}%
\right.  \label{ddtdiffeo2}
\end{equation}%
The ODE (\ref{ddtdiffeo2}) has a unique solution $\phi _{t}$. Now let 
\begin{equation}
\psi \left( t\right) =\left( \phi _{t}\right) ^{\ast }\bar{\psi}\left(
t\right) .
\end{equation}%
Since $\phi _{0}=Id$, 
\begin{equation*}
\psi \left( 0\right) =\psi _{0}\text{.}
\end{equation*}%
Also, 
\begin{equation*}
d\psi \left( t\right) =d\left( \left( \phi _{t}\right) ^{\ast }\bar{\psi}%
\left( t\right) \right) =\left( \phi _{t}\right) ^{\ast }d\bar{\psi}\left(
t\right) =0
\end{equation*}%
since diffeomorphisms commute with $d$ and $d\bar{\psi}\left( t\right) =0$.
Now let us show that $\psi \left( t\right) $ satisfies the flow equation (%
\ref{ivp1}). 
\begin{eqnarray}
\frac{d}{dt}\psi \left( t\right) &=&\frac{d}{dt}\left( \left( \phi
_{t}\right) ^{\ast }\bar{\psi}\left( t\right) \right)  \notag \\
&=&\left( \phi _{t}\right) ^{\ast }\left( \mathcal{L}_{-V\left( \bar{\chi}%
\right) }\bar{\psi}\left( t\right) \right) +\left( \phi _{t}\right) ^{\ast
}\left( \frac{d}{dt}\bar{\psi}\left( t\right) \right)  \label{ddtpsi}
\end{eqnarray}%
where we have used (\ref{ddtdiffeo2}). From (\ref{ivp3}), $\bar{\psi}\left(
t\right) $ satisfies the flow equation 
\begin{equation*}
\frac{d}{dt}\bar{\psi}\left( t\right) =\Delta _{\bar{\psi}}\bar{\psi}%
+2d\left( \left( A-\func{Tr}T_{\bar{\psi}}\right) \ast _{\bar{\psi}}\bar{\psi%
}\right) +\mathcal{L}_{V\left( \bar{\chi}\right) }\bar{\psi}
\end{equation*}%
Hence 
\begin{eqnarray*}
\left( \phi _{t}\right) ^{\ast }\left( \frac{d}{dt}\bar{\psi}\left( t\right)
\right) &=&\Delta _{\phi _{t}^{\ast }\bar{\psi}}\left( \phi _{t}^{\ast }\bar{%
\psi}\right) +2d\left( \left( A-\func{Tr}T_{\phi _{t}^{\ast }\bar{\psi}%
}\right) \ast _{\phi _{t}^{\ast }\bar{\psi}}\left( \phi _{t}^{\ast }\bar{\psi%
}\right) \right) +\left( \phi _{t}\right) ^{\ast }\left( \mathcal{L}%
_{V\left( \bar{\chi}\right) }\bar{\psi}\right) \\
&=&\Delta _{\psi }\psi +2d\left( \left( A-\func{Tr}T_{\psi }\right) \ast
_{\psi }\psi \right) +\left( \phi _{t}\right) ^{\ast }\left( \mathcal{L}%
_{V\left( \bar{\chi}\right) }\bar{\psi}\right)
\end{eqnarray*}%
So overall, from (\ref{ddtpsi}) we find 
\begin{equation*}
\frac{d}{dt}\psi \left( t\right) =\Delta _{\psi }\psi +2d\left( \left( A-%
\func{Tr}T_{\psi }\right) \ast _{\psi }\psi \right) .
\end{equation*}%
Thus indeed, $\psi \left( t\right) $ solves the system (\ref{ivp1}) for a
short time $0\leq t\leq \varepsilon $.

Now let us prove uniqueness. Suppose $\psi _{1}$ and $\psi _{2}$ are two
solutions of (\ref{ivp1}), and let $\phi _{1}$ and $\phi _{2}$ be
corresponding solutions of (\ref{ddtdiffeo}). For $i=1,2$ define 
\begin{equation}
\bar{\psi}_{i}=\left( \phi _{i}^{-1}\right) ^{\ast }\psi _{i}.
\end{equation}%
Both $\bar{\psi}_{i}$ are clearly closed, and satisfy 
\begin{eqnarray}
\frac{d}{dt}\bar{\psi}_{i}\left( t\right) &=&\left( \phi _{i}^{-1}\right)
^{\ast }\left( \mathcal{L}_{-\left( \phi _{i}^{-1}\right) _{\ast }V\left( 
\bar{\chi}_{i}\right) }\psi _{i}\left( t\right) \right) +\left( \phi
_{i}^{-1}\right) ^{\ast }\left( \frac{d}{dt}\psi _{i}\left( t\right) \right)
\\
&=&\mathcal{L}_{V\left( \bar{\chi}_{i}\right) }\left( \left( \phi
_{i}^{-1}\right) ^{\ast }\psi _{i}\right) +\left( \phi _{i}^{-1}\right)
^{\ast }\left( \Delta _{\psi _{i}}\psi _{i}+2d\left( \left( A-\func{Tr}%
T_{\psi _{i}}\right) \ast _{\psi _{i}}\psi _{i}\right) \right)  \notag \\
&=&\mathcal{L}_{V\left( \bar{\chi}_{i}\right) }\bar{\psi}_{i}+\Delta _{\bar{%
\psi}_{i}}\bar{\psi}_{i}+2d\left( \left( A-\func{Tr}T_{\bar{\psi}%
_{i}}\right) \ast _{\bar{\psi}_{i}}\bar{\psi}_{i}\right)  \notag
\end{eqnarray}%
This is precisely the flow equation (\ref{ivp3}). Now, $\phi _{i}\left(
0\right) =Id,$ and $\psi _{i}\left( 0\right) =\psi _{0},$ so $\bar{\psi}_{i}$
have same initial conditions. However, from Corollary \ref{corrivp3sol}, we
know that the system (\ref{ivp3}) has a unique solution, whence $\bar{\psi}%
_{1}=\bar{\psi}_{2}$. Then, $\phi _{1}$ and $\phi _{2}$ satisfy the same ODE
(\ref{ddtdiffeo2}) with the same initial conditions, so are also equal.
Therefore, $\psi _{1}=\psi _{2}$.
\end{proof}

\section{Concluding remarks}

We have thus found a modified Laplacian coflow of co-closed $G_{2}$%
-structures given by 
\begin{equation*}
\frac{d}{dt}\psi \left( t\right) =\Delta _{\psi }\psi +2d\left( \left( A-%
\func{Tr}T_{\psi }\right) \ast _{\psi }\psi \right)
\end{equation*}%
for a family of closed $4$-forms $\psi \left( t\right) $ and a constant $A$.
Given an initial condition $\psi \left( 0\right) =\psi _{0}$ this flow was
found to have a unique short-time existence. Moreover, if along the flow $%
\func{Tr}T_{\psi }$ remains non-negative and less than or equal to $\frac{4}{%
3}A$, the volume functional $V$ (\ref{volfuncpsi}) increases monotonically.
An alternative necessary and sufficient condition for this is a bound for
the total scalar curvature, in terms of $\func{Tr}T_{\psi }$:%
\begin{equation*}
\int_{M}R\mathrm{vol}\leq 2\int_{M}\func{Tr}T\left( 2A-\func{Tr}T\right) 
\mathrm{vol}
\end{equation*}%
Further questions can be asked about this flow, in particular, the long-term
existence of solutions. In general, one would expect singularities to
develop in finite time, but perhaps there are some initial conditions which
lead to smooth long-term solutions. A more likely property is the stability
of solutions - if the initial condition is sufficiently close to a
torsion-free $G_{2}$-structure, whether that would lead to a long-term
solution. It is possible that the method used by Xu and Ye in \cite{XuYe}
for the Laplacian flow of $\varphi $ could be adapted in this scenario. This
will be the subject of further study. Another interesting question is
whether given solutions of the modified flow (\ref{ivp1}) one can find
solutions of the original Laplacian coflow of $\psi $ (\ref{orig-coflow}).
The answers to these questions should lead to a better understanding of the
relationships between different torsion classes of $G_{2}$-structures, and
torsion-free $G_{2}$-structures, in particular.

\bibliographystyle{jhep-a}
\bibliography{refs2}

\end{document}